\documentclass[final]{siamltex}
\usepackage[top=1.45in, bottom=1.45in, left=1.5in, right=1.5in]{geometry}

\newif\ifpdf\ifx\pdfoutput\undefined\pdffalse\else\pdfoutput=1\pdftrue\fi
\usepackage[backref,colorlinks=true, linkcolor=blue, citecolor=blue, urlcolor=blue, pdfborder={0 0 0}]{hyperref}

  \ifpdf
    \hypersetup{pdftitle={EXTRA: An Exact First-Order Algorithm for Decentralized Consensus Optimization}, pdfauthor={Wei Shi, Qing Ling, Gang Wu, Wotao Yin}}
        \usepackage[pdftex]{graphicx}
        \usepackage{url}
  \else\usepackage{graphicx}\fi

\usepackage{amsmath, amssymb, bm}
\usepackage{cases}
\usepackage{float, subfigure}
\usepackage{color}
\usepackage{amsfonts}
\usepackage{cite, url}
\usepackage[all,cmtip]{xy}

\usepackage{color,hyperref}
\definecolor{darkblue}{rgb}{0.0,0.0,0.6}
\hypersetup{colorlinks,breaklinks,
            linkcolor=darkblue,urlcolor=darkblue,
            anchorcolor=darkblue,citecolor=darkblue}

\newcommand{\highlight}[1]{{#1}}  
\usepackage[normalem]{ulem} 
\newcommand{\spa}[1]{{\mathrm{span}\{#1\}}}
\newcommand{\nul}[1]{{\mathrm{null}\{#1\}}}
\newtheorem{remark}{Remark}
\newtheorem{assumption}{Assumption}

\newcommand{\R}{{\mathbb{R}}}
\newcommand{\bx}{{\mathbf{x}}}
\newcommand{\bu}{{\mathbf{u}}}
\newcommand{\bv}{{\mathbf{v}}}
\newcommand{\bq}{{\mathbf{q}}}
\newcommand{\bz}{{\mathbf{z}}}
\newcommand{\f}{{\mathbf{f}}}
\newcommand{\g}{{\mathbf{g}}}
\newcommand{\df}{{\nabla\mathbf{f}}}
\newcommand{\dg}{{\nabla\mathbf{g}}}
\newcommand{\W}{{W}}
\newcommand{\BW}{{\tilde{W}}}
\newcommand{\bw}{{\tilde{w}}}
\newcommand{\one}{{\mathbf{1}}}
\newcommand{\Fro}{{\mathrm{F}}}
\newcommand{\Lf}{{L_{\f}}}
\newcommand{\mug}{{\mu_{\mathbf{g}}}}
\newcommand{\mubarf}{{\mu_{\bar{f}}}}
\newcommand{\trace}{{\mathrm{trace}}}
\newcommand{\T}{{\mathrm{T}}}
\newcommand{\degree}{{\mathrm{deg}}}
\DeclareMathOperator*{\Min}{minimize}

\title{EXTRA: An Exact First-Order Algorithm for Decentralized Consensus Optimization\thanks{This work is supported by Chinese Scholarship Council (CSC) grants 201306340046 and 2011634506, NSFC grant 61004137, MOF/MIIT/MOST grant BB2100100015, and NSF grants DMS-0748839 and DMS-1317602.}}

\author{Wei Shi \and Qing Ling \and Gang Wu \and Wotao Yin}

\begin{document}

\maketitle
\begin{abstract}
Recently, there have been growing interests in solving  consensus
optimization problems in a multi-agent network. In this paper, we
develop a decentralized algorithm for the consensus optimization
problem
$$\Min\limits_{x\in\R^p}~\bar{f}(x)=\frac{1}{n}\sum\limits_{i=1}^n f_i(x),$$
which is defined over a connected network of $n$ agents, where
each function $f_i$ is held privately by agent $i$ and encodes the agent's
data and objective. All the agents shall collaboratively find the
minimizer while each agent can only communicate with its
neighbors. Such a computation scheme avoids a data fusion center
or long-distance communication and offers better load balance to
the network.


This paper proposes a novel decentralized \uline{ex}act
firs\uline{t}-orde\uline{r} \uline{a}lgorithm (abbreviated as
EXTRA) to solve the consensus optimization problem. ``Exact''
means that it can converge to the exact solution. EXTRA can use a
fixed  large step size, {which is independent of the network size}, and has synchronized iterations. The
local variable of every agent $i$ converges uniformly and
consensually to an exact minimizer of $\bar{f}$. In contrast, the
well-known decentralized gradient descent (DGD) method must use
diminishing step sizes in order to converge to an exact minimizer.
EXTRA and DGD have the same choice of mixing matrices and similar per-iteration complexity.  EXTRA, however, uses
the gradients of last two iterates, unlike DGD which uses just  that of last iterate.

EXTRA has the best known convergence {rates} among the existing
first-order decentralized algorithms for decentralized consensus
optimization with convex Lipschitz--differentiable objectives.
Specifically, if $f_i$'s are convex and have Lipschitz continuous
gradients, EXTRA has an ergodic convergence rate
$O\left(\frac{1}{k}\right)$ in terms of the first-order optimality
residual. If $\bar{f}$ is also (restricted) strongly convex,
{EXTRA converges to an optimal solution at a linear rate
$O(C^{-k})$ for some constant $C>1$.}
\end{abstract}



\section{Introduction}\label{sec:intro}

This paper focuses on \textit{decentralized consensus
optimization}, a problem defined on a connected
network and solved by   $n$ agents  cooperatively
\begin{align} \label{eq:F}
\Min\limits_{x\in\R^p}~\bar{f}(x)=\frac{1}{n}\sum\limits_{i=1}^n
f_i(x),
\end{align}
over a common variable $x\in\R^p$, and for each agent
$i$, $f_i: \R^p\rightarrow \R$ is a convex function privately
known by the agent. We assume that $f_i$'s are continuously
differentiable and will introduce a novel first-order algorithm to
solve \eqref{eq:F} in a decentralized manner. We stick to the
synchronous case in this paper, that is, all the agents carry out
their iterations at the same time intervals.

Problems of the form (\ref{eq:F}) that require decentralized
computation are found widely in various scientific and engineering
areas including sensor network information processing,
multiple-agent control and coordination, as well as distributed
machine learning. Examples and works include decentralized
averaging \cite{Dimakis2010,Johansson2008,Xiao2007}, learning
\cite{Forero2010,Mateos2010,Predd2009}, estimation
\cite{Bazerque2010,Bazerque2011,Kekatos2013,Ling2010,Schizas2008},
sparse optimization \cite{Ling2013_2,Kun2013}, and low-rank matrix
completion \cite{Ling2012_2} problems. Functions $f_i$ can take
forms of least squares \cite{Dimakis2010,Johansson2008,Xiao2007},
regularized least squares
\cite{Bazerque2010,Bazerque2011,Forero2010,Ling2010,Mateos2010},
as well as more general ones \cite{Predd2009}. The solution $x$
can represent, for example, the average temperature of a room
\cite{Dimakis2010,Xiao2007}, frequency-domain occupancy of spectra
\cite{Bazerque2010,Bazerque2011}, states of a smart grid system
\cite{Gan2013,Kekatos2013}, sparse vectors
\cite{Ling2013_2,Kun2013}, and a matrix factor \cite{Ling2012_2}
and so on. In general, decentralized optimization fits the
scenarios in which the data is collected and/or stored in a
distributed network, a fusion center is either infeasible or not
economical, and/or computing is required to be performed in a
decentralized and collaborative manner by multiple agents.

\subsection{Related Methods}

Existing first-order decentralized methods for solving
(\ref{eq:F}) include the (sub)gradient method
\cite{Matei2011,Nedic2009,Kun2014}, the (sub)gradient-push method
\cite{Nedic2013,Nedic2014}, the fast (sub)gradient method
\cite{Chen2012,Jakovetic2014}, and the dual averaging method
\cite{Duchi2012}. Compared to classical centralized algorithms,
decentralized algorithms encounter more restrictive assumptions
and typically worse convergence rates. {Most of the
above algorithms are analyzed under the assumption of bounded
(sub)gradients. \highlight{Work \cite{Matei2011} assumes bounded Hessian for
strongly convex functions.} Recent work \cite{Kun2014} relaxes such
assumptions for decentralized gradient descent. \highlight{When (\ref{eq:F})
has additional constraints that force $x$ in a bounded set, which
also leads to bounded (sub)gradients and Hessian, projected
first-order algorithms are applicable \cite{Ram2010,Zhu2012}.}}


When using a fixed step size, these algorithms do not converge to
a solution $x^*$ of problem \eqref{eq:F} but a point in its
neighborhood no matter whether $f_i$'s are differentiable or not
\cite{Kun2014}. This motivates the use of certain diminishing step
sizes in \cite{Chen2012,Duchi2012,Jakovetic2014} to guarantee
convergence to $x^*$. The rates of convergence are generally
weaker than their analogues in centralized computation.
\highlight{For the general convex case and under the bounded
(sub)gradient (or Lipschitz--continuous objective) assumption,
\cite{Chen2012} shows that diminishing step sizes
$\alpha_k=\frac{1}{\sqrt{k}}$  lead to  a convergence rate of
$O\left(\frac{\ln{k}}{\sqrt{k}}\right)$ in terms of the running
best of objective error, and  \cite{Duchi2012} shows that the dual
averaging method  has a rate of $O\left(\frac{\ln
k}{\sqrt{k}}\right)$ in the ergodic sense in terms of objective
error. For the general convex case, under assumptions of fixed
step size and Lipschitz continuous, bounded gradient,
\cite{Jakovetic2014} shows an outer--loop convergence rate of
$O\left(\frac{1}{k^2}\right)$ in terms of objective error,
utilizing Nesterov's acceleration, provided that the inner loop
performs substantial consensus computation, without which
diminishing step sizes $\alpha_k=\frac{1}{k^{1/3}}$ lead to a
reduced rate of $O\left(\frac{\ln{k}}{k}\right)$. The
(sub)gradient-push method  \cite{Nedic2013} can be implemented in
a dynamic digraph and, under the bounded (sub)gradient assumption
and diminishing step sizes $\alpha_k =
O\left(\frac{1}{\sqrt{k}}\right)$, has a rate of
$O\left(\frac{\ln{k}}{\sqrt{k}}\right)$ in the ergodic sense in
terms of  objective error. A better
 rate of  $O\left(\frac{\ln{k}}{k}\right)$ is  proved for the (sub)gradient-push method in \cite{Nedic2014} under the strong convexity  and Lipschitz gradient assumptions, in terms of expected objective error plus squared consensus residual.

Some of other related algorithms are as follows.
For general
convex functions and assuming closed and bounded
feasible sets, the decentralized asynchronous ADMM \cite{Wei2013}
is proved to have a rate of $O\left(\frac{1}{k}\right)$  in terms
of  expected objective error and feasibility violation. The
augmented Lagrangian based primal-dual methods  have linear
convergence under strong convexity and Lipschitz gradient
assumptions \cite{Chang2014,Shi2014} or under the
positive-definite bounded Hessian assumption
\cite{Iutzeler2013,Jakovetic2013}.


Our proposed algorithm is a synchronous gradient-based algorithm
that has a rate of $O\left(\frac{1}{k}\right)$ for general convex
objectives with Lipschitz differentials and has a linear rate once the sum of, rather than
individual, functions $f_i$ is also (restricted) strongly convex.}


\subsection{Notation}

Throughout the paper, we let agent $i$ hold a \emph{local copy} of
the global variable $x$, which {is} denoted by
$x_{(i)}\in\R^p$; its value at iteration $k$ is denoted by
$x_{(i)}^{k}$. We introduce an aggregate objective function of the
local variables
$$
\f(\bx)\triangleq\sum\limits_{i=1}^{n} f_i(x_{(i)}),
$$
where
$$
  \bx\triangleq\left(
     \begin{array}{ccc}
       \textrm{---}& x_{(1)}^\T & \textrm{---} \\
       \textrm{---}& x_{(2)}^\T & \textrm{---} \\
       &\vdots& \\
       \textrm{---}& x_{(n)}^\T & \textrm{---} \\
     \end{array}
   \right)\in\R^{n\times p}.
$$
The gradient of $\f(\bx)$ is  defined by
$$
   \df(\bx)\triangleq\left(
     \begin{array}{ccc}
       \textrm{---}& \nabla^\T f_1(x_{(1)}) & \textrm{---} \\
       \textrm{---}& \nabla^\T f_2(x_{(2)}) & \textrm{---} \\
       &\vdots & \\
       \textrm{---}& \nabla^\T f_n(x_{(n)}) & \textrm{---} \\
     \end{array}
   \right)\in\R^{n\times p}.
$$
Each row $i$ of $\bx$ and $\df(\bx)$ is associated with agent $i$. We say that $\bx$ is \emph{consensual} if all of its rows are identical, i.e., $x_{(1)}=\cdots= x_{(n)}$. 
The analysis and results of this paper hold for all $p\geq1$. The reader can assume $p=1$ for convenience (so $\bx$ and $\df$ become vectors) without missing any major point.

Finally, for given matrix $A$ and symmetric positive semidefinite matrix $G$,
we define the $G$-matrix norm $\|A\|_G\triangleq\sqrt{\trace(A^\T
GA)}$. The largest singular value of a matrix $A$ is denoted as
$\sigma_{\max}(A)$. The largest and smallest eigenvalues of a
symmetric matrix $B$ are denoted as $\lambda_{\max}(B)$ and
$\lambda_{\min}(B)$, respectively. The smallest \emph{nonzero} eigenvalue
of a symmetric positive semidefinite matrix $B\not=\mathbf{0} $ is denoted as
$\tilde{\lambda}_{\min}(B)$, which is strictly positive. For a matrix $A\in\R^{m\times n}$,
 $\nul{A}\triangleq\{x\in\R^n\big|Ax=0\}$ is the null space of
 $A$ and $\spa{A}\triangleq \{y\in\R^m\big|y=Ax,\forall
x\in\R^n\}$ is the linear span of all the columns of $A$.

\subsection{Summary of Contributions}

This paper introduces a novel gradient-based decentralized
algorithm EXTRA, establishes its convergence conditions and rates, and
presents  numerical results in comparison to
decentralized gradient descent. EXTRA can use a fixed step size independent of the network size
and quickly converges to the solution to \eqref{eq:F}. {It has a rate of convergence $O\left(\frac{1}{k}\right)$
in terms of   best running violation to the first-order optimality condition when $\bar{f}$ is Lipschitz differentiable, and has a
linear rate of convergence if $\bar{f}$ is also (restricted) strongly
convex.} Numerical simulations verify the theoretical results
and demonstrate its competitive performance.

\subsection{Paper Organization}

The rest of this paper is organized as follows. Section
\ref{sec:algo} {develops and interprets EXTRA}. Section
\ref{sec:conv} presents its convergence results. Then, Section
\ref{sc:num} presents three sets of numerical results. Finally,
Section \ref{sc:concl} concludes this paper.

\section{Algorithm Development}\label{sec:algo}

This section derives the proposed algorithm EXTRA. We start by
briefly reviewing \emph{decentralized gradient descent} (DGD) and
discussing the dilemma that  DGD  converges slowly to an exact
solution  when it uses a sequence of diminishing step sizes, yet
it converges faster using a fixed step size but stalls at an
inaccurate solution. We then obtain the update formula of EXTRA by
taking the difference of two formulas of the DGD update. Provided
that the sequence generated by the new update formula  with a
fixed step size converges to a point, we argue that the point is
consensual and optimal. Finally, we briefly discuss the choice of
mixing matrices in EXTRA. Formal convergence results and proofs
are left to Section \ref{sec:conv}.

\subsection{Review of Decentralized Gradient Descent and Its Limitation}\label{sec:dgd}

DGD carries out the following iteration
\begin{equation} \label{eq:DGD}
\begin{array}{cl}
x_{(i)}^{k+1} = \sum\limits_{j=1}^n w_{ij} x_{(j)}^{k} - \alpha^k
\nabla f_i(x_{(i)}^{k}),\quad \mbox{for agent}~i=1,\ldots,n.
\end{array}
\end{equation}
Recall that $x_{(i)}^{k} \in \R^p$ is the local copy of $x$ held
by agent $i$ at iteration $k$, $W=[w_{ij}]\in\R^{n\times n}$ is a symmetric mixing matrix satisfying
$\nul{I-W}=\spa{\mathbf{1}}$ and
$\sigma_{\max}(W-\frac{1}{n}\mathbf{1}\mathbf{1}^\T)<1$, and
$\alpha^k>0$ is a step size for iteration $k$. If two agents $i$
and $j$ are neither neighbors nor identical, then $w_{ij}=0$. This
way, the computation of \eqref{eq:DGD} involves only local and
neighbor information, and hence the iteration is decentralized.

Following our notation, we rewrite \eqref{eq:DGD} for all the agents together as
\begin{equation} \label{eq:DGD_matrix}
\begin{array}{cl}
\bx^{k+1} = W\bx^{k} - \alpha^k\df(\bx^{k}).
\end{array}
\end{equation}
With a fixed step size $\alpha^k\equiv\alpha$, DGD has
\emph{inexact convergence}. For each agent $i$, $x_{(i)}^k$
converges to a point in the $O(\alpha)$-neighborhood of a solution
to \eqref{eq:F}, and these points for different agents can be
different. On the other hand, properly reducing $\alpha^k$ enables
\emph{exact convergence}, namely, that each $x^k_{(i)}$ converges
to the same exact solution. However,  reducing $\alpha^k$ causes
slower convergence, both in theory and in practice.

Paper \cite{Kun2014} assumes that $\nabla f_i$'s are  Lipschitz
continuous, and studies DGD with a constant
$\alpha^{k}\equiv\alpha$. Before the iterates reach the
$O(\alpha)$-neighborhood,  the objective value reduces at the rate
$O\left(\frac{1}{k}\right)$, and this rate improves to linear if
$f_i$'s are also (restricted) strongly convex. In comparison,
paper \cite{Jakovetic2014} studies DGD with diminishing
$\alpha^k=\frac{1}{k^{1/3}}$ and assumes that $\nabla f_i$'s are
Lipschitz continuous and bounded. The objective convergence rate
slows down to $O\left(\frac{1}{k^{2/3}}\right)$. Paper
\cite{Chen2012} studies DGD with diminishing
$\alpha^k=\frac{1}{k^{1/2}}$ and assumes that $f_i$'s are
Lipschitz continuous; a slower rate
$O\left(\frac{\ln{k}}{\sqrt{k}}\right)$ is proved. A simple
example of decentralized least squares in Section \ref{sec:DLS}
gives a rough comparison of these three schemes (and how they
compare to the proposed algorithm).

To see the cause of \emph{inexact convergence} with a
\emph{fixed step size},  let $\bx^{\infty}$ be the limit of
$\bx^k$ (assuming the step size is small enough to ensure convergence).
Taking the limit over $k$ on both sides of iteration
\eqref{eq:DGD_matrix} gives us
$$
\bx^{\infty} = W\bx^{\infty} - \alpha\df(\bx^{\infty}).
$$
When $\alpha$ is fixed and nonzero, assuming the consensus of $\bx^{\infty}$ (namely, it has identical rows $x^{\infty}_{(i)}$) will mean   $\bx^{\infty} = W\bx^{\infty}$, as a result of
$W\mathbf{1}=\mathbf{1}$, and thus
$\df(\bx^{\infty})=\mathbf{0}$, which is equivalent to $\nabla
f_i(x^{\infty}_{(i)})=0,~\forall i$, i.e., the same point  $x^{\infty}_{(i)}$ simultaneously minimizes  $f_i$ for all
agents $i$. This is impossible in general and is different from our objective to find a point that minimizes   $\sum_{i=1}^n
f_i$. 

\subsection{Development of EXTRA}

The next proposition  provides simple conditions  for the consensus and optimality for problem \eqref{eq:F}.

\begin{proposition}\label{conopt} Assume $\nul{I-W}=\spa{\mathbf{1}}$. If
\begin{equation}\label{def:sol}
  \bx^*\triangleq\left(
     \begin{array}{ccc}
       \textrm{---}& x_{(1)}^{*\T} & \textrm{---} \\
       \textrm{---}& x_{(2)}^{*\T} & \textrm{---} \\
       &\vdots& \\
       \textrm{---}& x_{(n)}^{*\T} & \textrm{---} \\
     \end{array}
   \right)
\end{equation}
satisfies conditions:
\begin{enumerate}
\item $\bx^*=W\bx^*$ (consensus),
\item $\one^\T\df(\bx^*)=0$ (optimality),
\end{enumerate}
then $x^*=x_{(i)}^*$, for any $i$, is a solution to the consensus
optimization problem \eqref{eq:F}.
\end{proposition}

\begin{proof}
Since $\nul{I-W}=\spa{\mathbf{1}}$, $\bx$ is consensual if and
only if condition 1 holds, i.e., $\bx^*=W\bx^*$. Since $\bx^*$
is consensual, we have $\one^\T\df(\bx^*)=\sum_{i=1}^n\nabla f_i(x^*)$, so  condition 2 means optimality. \hfill
\end{proof}

Next, we construct the update formula of EXTRA, following which the iterate sequence will converge to a point satisfying the two conditions in Proposition \ref{conopt}.

\highlight{{Consider the DGD update \eqref{eq:DGD_matrix} written at iterations $k+1$ and $k$  as follows
\begin{align}
\bx^{k+2} &= W\bx^{k+1} - \alpha\df(\bx^{k+1}),\label{dgdk1}\\
\bx^{k+1} &= \BW\bx^{k} - \alpha\df(\bx^{k}),\label{dgdk}
\end{align}
where the former uses the mixing matrix $W$ and the latter uses
$$\BW=\frac{I+W}{2}.$$ {The choice of $\BW$ will be generalized
later.} The update formula of EXTRA is simply their difference,
subtracting \eqref{dgdk} from \eqref{dgdk1}:}
\begin{equation}\label{cw}
\bx^{k+2}-\bx^{k+1} = W\bx^{k+1}-\BW\bx^{k} -
\alpha\df(\bx^{k+1})+\alpha\df(\bx^{k}).
\end{equation}
Given $\bx^k$ and $\bx^{k+1}$, the next iterate  $\bx^{k+2}$ is generated by \eqref{cw}.


Let us assume that $\{\bx^k\}$ converges for now and let $\bx^*
=\lim_{k\rightarrow\infty} \bx^k$. Let us also assume that $\df$
is continuous. We first establish condition 1 of Proposition
\ref{conopt}.  Taking $k\rightarrow\infty$ in \eqref{cw} gives us
\begin{equation}\label{cw1}
\bx^* - \bx^* =(W-\BW)\bx^*-\alpha\df(\bx^*)+\alpha\df(\bx^*),
\end{equation}
from which  it follows that
\begin{equation}\label{cw2}
W\bx^*-\bx^*=2(W-\BW)\bx^*=\mathbf{0}.
\end{equation}
{Therefore, $\bx^*$  is consensual.}

{Provided that $\one^\T(W-\BW)=0$, we show that
$\bx^*$ also satisfies condition 2 of Proposition
\ref{conopt}.} To see this, adding the first update $\bx^1=W\bx^0-\alpha\df(x^0)$
to the subsequent updates following the formulas of
$(\bx^{2}-\bx^{1}),(\bx^{3}-\bx^{2}),\ldots,(\bx^{k+2}-\bx^{k+1})$
given by \eqref{cw} and then applying telescopic cancellation,  we obtain
\begin{equation}\label{cw3}
\bx^{k+2}=\BW\bx^{k+1}-\alpha\df(\bx^{k+1})+\sum\limits_{t=0}^{k+1}(W-\BW)\bx^t,
\end{equation}
or equivalently,
\begin{equation}\label{cw3_2}
\bx^{k+2}=\W\bx^{k+1}-\alpha\df(\bx^{k+1})+\sum\limits_{t=0}^{k}(W-\BW)\bx^t.
\end{equation}
Taking $k\rightarrow\infty$,  from
$\bx^*=\lim_{k\rightarrow\infty} \bx^k$ and
$\bx^*=\BW\bx^*=\W\bx^*$, it follows that
\begin{equation}\label{cw4}
\alpha\df(\bx^*)=\sum\limits_{t=0}^\infty(W-\BW)\bx^t.
\end{equation}
Left-multiplying $\one^\T$ on both sides of \eqref{cw4}, in light of $\one^\T(W-\BW)=0$, we obtain the condition 2 of Proposition
\ref{conopt}:
\begin{equation}\label{cw5}
\one^\T\df(\bx^*)=0.
\end{equation}

{To summarize, provided that $\nul{I-W}=\spa{\mathbf{1}}$, $\BW=\frac{I+W}{2}$, $\one^\T(W-\BW)=0$, and the continuity of $\df$, if a sequence following EXTRA \eqref{cw} converges to a point $\bx^*$, then by Proposition
\ref{conopt}, $\bx^*$ is  consensual and any of its identical row vectors solves problem \eqref{eq:F}. 
}
}
\subsection{The Algorithm EXTRA and its Assumptions}

We present  EXTRA --- an
\uline{ex}act firs\uline{t}-orde\uline{r} \uline{a}lgorithm for
decentralized consensus optimization --- in Algorithm 1.

\begin{center}
  {\textbf{Algorithm 1: EXTRA}}

  \smallskip
    \begin{tabular}{l}
    \hline
    \emph{  } Choose $\alpha>0$ and mixing matrices $W\in \R^{n \times n}$ and $\BW\in \R^{n \times n}$;\\
    \emph{  } Pick any $\bx^0\in \R^{n \times p}$;\\
    \emph{1.} $\bx^1\gets W\bx^0-\alpha\df(\bx^0)$;\\
    \emph{2.} \textbf{for} $k=0,1,\cdots$ \textbf{do} \\
    \qquad$\bx^{k+2}\gets (I+\W)\bx^{k+1}-\BW\bx^k-\alpha\left[\df(\bx^{k+1})-\df(\bx^k)\right]$;\\
    \emph{~ } \textbf{end for}\\
    \hline
    \end{tabular}
\end{center}
\vspace{10pt}
Breaking to the individual agents, Step 1 of EXTRA performs updates
$$
x_{(i)}^{1} = \sum\limits_{j=1}^n w_{ij} x_{(j)}^{0} - \alpha \nabla f_i(x_{(i)}^{0}),\quad i=1,\ldots, n,
$$
and Step 2 at each iteration $k$ performs updates
$$
x_{(i)}^{k+2}= x_{(i)}^{k+1}+\sum\limits_{j=1}^nw_{ij}x_{(j)}^{k+1}-\sum\limits_{j=1}^n\bw_{ij}x_{(j)}^k-\alpha\left[\nabla f_i(x_{(i)}^{k+1})-\nabla f_i(x_{(i)}^k)\right],\quad i=1,\ldots, n.
$$
Each agent computes $\nabla f_i(x_{(i)}^k)$ once for each $k$ and uses it twice for $x_{(i)}^{k+1}$ and $x_{(i)}^{k+2}$. For our recommended choice of $\tilde W=(W+I)/2$, each agent computes $\sum_{j=1}^n w_{ij}x_{(j)}^{k}$ once as well.

Here we formally give the assumptions on the mixing matrices $W$ and $\tilde W$ for  EXTRA. All of them will be used in the convergence analysis in the next section.
\begin{assumption}[Mixing matrix]\label{ass:matrices} Consider a \emph{connected network}
$\mathcal{G}=\{\mathcal{V}, \mathcal{E}\}$ consisting of a set of
agents $\mathcal{V}=\{1,2,\cdots,n\}$ and a set of undirected
edges $\mathcal{E}$. 
The mixing matrices $\W=[w_{ij}]\in\R^{n\times n}$ and
$\BW=[\bw_{ij}]\in\R^{n\times n}$ satisfy
\begin{enumerate}
\item (Decentralized property) If $i\not=j$ and $(i,j)\not\in\mathcal{E}$, then
$w_{ij}=\bw_{ij}=0$.
\item (Symmetry) $\W = \W^\T$, $\BW = \BW^\T$.
\item (Null space property) $\nul{\W-\BW}=\spa{\mathbf{1}}$, $\nul{I-\BW}\supseteq\spa{\mathbf{1}}$.
\item (Spectral
property)  $\BW \succ 0$ and
$\frac{I+\W}{2}\succcurlyeq\BW\succcurlyeq\W$.
\end{enumerate}
\end{assumption}

We claim that Parts 2--4 of Assumption \ref{ass:matrices}
imply $\nul{I-W}=\spa{\mathbf{1}}$ and the eigenvalues of $\W$ lie in
$(-1,1]$, which are commonly assumed for DGD. Therefore, the additional assumptions are merely on $\tilde W$. In fact, EXTRA can use the same $W$ used in DGD and simply take $\BW=\frac{I+W}{2}$, which satisfies Part 4. 
It is also worth noting that the recent work push-DGD \cite{Nedic2013} relaxes the symmetry condition, yet such relaxation for EXTRA is not trivial and is our future work.

\highlight{\begin{proposition} Parts 2--4 of Assumption \ref{ass:matrices} imply $\nul{I-W}=\spa{\mathbf{1}}$ and that the eigenvalues of $\W$ lie in
$(-1,1]$.
\end{proposition}
\begin{proof}
From part 4, we have $\frac{I+\W}{2}\succcurlyeq\BW\succ 0$ and
thus $W\succ - I$ and {$\lambda_{\min}(W)>-1$}. Also from
part 4, we have $\frac{I+\W}{2}\succcurlyeq\W$ and thus $I\succeq
W$, which means $\lambda_{\max}(W)\le 1$. Hence, all eigenvalues
of $\W$ (and those of $\BW$) lie in $(-1,1]$.

Now, we show
$\nul{I-\W}=\spa{\mathbf{1}}$. Consider a \emph{nonzero }vector $\bv\in\nul{I-\W}$, which satisfies $(I-\W)\bv=0$ and thus $\bv^T(I-W)\bv=0$ and  $\bv^T\bv=\bv^T\W\bv$. From $\frac{I+\W}{2}\succcurlyeq\BW$ (part 4), we get $\bv^T\bv=\bv^T(\frac{I+\W}{2})\bv \ge \bv^T\BW\bv$, while from $\BW\succcurlyeq\W$ (part 4) we also get
$\bv^T\BW\bv\ge \bv^T\W\bv = \bv^T\bv$. Therefore, we have $\bv^T\BW\bv=\bv^T\bv$ or equivalently $(\BW-I)\bv=0$, adding which to $(I-\W)\bv=0$ yields $(\BW-\W)\bv=0$. In light of $\nul{\W-\BW}=\spa{\mathbf{1}}$ (part 3), we must have $\bv \in \spa{\mathbf{1}}$ and thus $\nul{I-\W}=\spa{\mathbf{1}}.$
\hfill \end{proof}
}

\subsection{Mixing Matrices}\label{sec:matrices}
In EXTRA, the mixing matrices $W$ and $\BW$
diffuse information throughout the network.

The role of $W$ is the similar as that in DGD
\cite{Chen2012,Tsitsiklis1984,Kun2014} and average consensus
\cite{Xiao2004}. It has a few common choices, which can
significantly affect performance.
\begin{enumerate}
  \item[(i)] Symmetric doubly stochastic matrix \cite{Chen2012,Tsitsiklis1984,Kun2014}: $W=W^\T$, $W\one=\one$, and $w_{ij}\geq0$. Special cases of such matrices include parts (ii) and (iii) below.
  \item[(ii)] Laplacian-based constant edge weight matrix \cite{Sayed2012,Xiao2004},
      $$
      W=I-\frac{L}{\tau},
      $$
      where $L$ is the Laplacian matrix of the graph $\mathcal{G}$ and $\tau>\frac{1}{2}\lambda_{\max}(L)$ is a scaling parameter. Denote $\degree(i)$ as the degree of agent $i$. When $\lambda_{\max}(L)$ is not available, $\tau=\max_{i\in\mathcal{V}}\{\degree(i)\}+\epsilon$ for some small  $\epsilon>0$, say $\epsilon=1$, can be used.
  \item[(iii)] Metropolis constant edge weight matrix \cite{Boyd2004,Xiao2007},
      $$
      w_{ij}=\left\{
      \begin{array}{cl}
      \frac{1}{\max\{\degree(i),\degree(j)\}+\epsilon},&\text{if } (i,j)\in\mathcal{E}, \\
      0,&\text{if } (i,j)\notin\mathcal{E}\text{ and } i\neq j, \\
      1-\sum\limits_{k\in\mathcal{V}} w_{ik},&\text{if } i=j,\\
      \end{array}
      \right.
      $$
      for some small positive $\epsilon>0$.
  \item[(iv)] Symmetric fastest distributed linear averaging (FDLA) matrix. It is a symmetric $W$ that achieves fastest information diffusion and can be obtained by  a semidefinite program \cite{Xiao2004}.
\end{enumerate}

It is worth noting that the optimal choice for average consensus, FDLA, no longer appears optimal in  decentralized consensus optimization, which is more general.

When $W$ is chosen following any strategy above,  $\BW=\frac{I+W}{2}$ is found to be very efficient.

\subsection{EXTRA as Corrected DGD}

We rewrite \eqref{cw3_2} as
\begin{equation} \label{eq:EXTRA_interpretation}
\begin{array}{cl}
\underbrace{\bx^{k+1}=W\bx^{k}-\alpha\df(\bx^{k})}_{\text{DGD}}+\underbrace{\sum\limits_{t=0}^{k-1}(W-\BW)\bx^t}_{\text{correction}},\quad
k=0,1,\cdots.
\end{array}
\end{equation}
An EXTRA update is, {therefore}, a DGD update with a cumulative
correction term. In subsection \ref{sec:dgd}, we have argued that
the DGD update cannot reach consensus asymptotically unless
$\alpha$ asymptotically vanishes. Since $\alpha\df(\bx^k)$ with a
fixed $\alpha>0$ cannot vanish in general, it must be corrected,
or otherwise $\bx^{k+1}-W\bx^k$ does not vanish, preventing
$\bx^k$ from being asymptotically consensual. Provided that
\eqref{eq:EXTRA_interpretation} converges, the role of the
cumulative term $\sum_{t=0}^{k-1}(W-\BW)\bx^t$ is to
\emph{neutralize} $-\alpha\df(\bx^{k})$ in
$(\spa{\mathbf{1}})^\perp$, the subspace  orthogonal to
$\mathbf{1}$. If a vector $\mathbf{v}$ obeys
$\mathbf{v}^\T(W-\BW)=0$, then the convergence {of}
\eqref{eq:EXTRA_interpretation} means the vanishing of
$\mathbf{v}^\T\df(\bx^k)$ in the limit. We need
$\one^\T\df(\bx^k)=0$ for consensus optimality. The correction
term in \eqref{eq:EXTRA_interpretation} is the simplest that we
could find so far. In particular, the summation is necessary since
each individual term  $(W-\BW)\bx^t$ is asymptotically vanishing.
The terms must work cumulatively.

\section{Convergence Analysis}\label{sec:conv}

To establish convergence of EXTRA, this paper makes two additional
but common assumptions as follows. Unless otherwise stated, the
results in this section are given under Assumptions
\ref{ass:matrices}--\ref{ass:solotion_existence}.

\begin{assumption}\label{ass:functions} \textbf{(Convex objective with Lipschitz continuous gradient)}
Objective functions $f_i$ are proper closed convex and Lipschitz
differentiable:
$$\|\nabla f_i(x_a) - \nabla f_i(x_b)\|_2 \leq
L_{f_i} \|x_a - x_b\|_2,\quad \forall x_a, x_b \in \R^p,$$ where
$L_{f_i} \ge 0$ are constant.
\end{assumption}

Following Assumption \ref{ass:functions}, function
$\f(\bx)=\sum_{i=1}^{n} f_i(x_{(i)})$ is proper closed convex, and
$\df$ is Lipschitz continuous
$$
\|\df(\bx_a)-\df(\bx_b)\|_\Fro \leq
\Lf\|\bx_a-\bx_b\|_\Fro,\quad\forall\bx_a, \bx_b \in \R^{n\times
p},
$$
with constant $\Lf = \max_i \{L_{f_i}\}.$

\begin{assumption}\label{ass:solotion_existence}
\textbf{(Solution existence)} Problem \eqref{eq:F} has a nonempty
set of optimal solutions: $\mathcal{X}^*\neq\emptyset$.
\end{assumption}

\subsection{Preliminaries}

\highlight{We first state a lemma that gives the first-order
optimality conditions of \eqref{eq:F}.}

\begin{lemma}[First-order optimality conditions] \label{lemma:opt}
Given mixing matrices $\W$ and $\BW$, define $U = (\BW-W)^{1/2}$
by letting  $U\triangleq V S^{1/2} V^\T\in\R^{n \times n}$ where
$VSV^\T=\BW - \W$ is the economical-form singular value
decomposition. Then, under  Assumptions
\ref{ass:matrices}--\ref{ass:solotion_existence}, $\bx^*$ is
consensual and $x_{(1)}^* \equiv x_{(2)}^* \equiv \cdots \equiv
x_{(n)}^*$ is optimal to problem \eqref{eq:F} if and only if there
exists $\bq^*=U\mathbf{p}$ for some $\mathbf{p} \in \R^{n\times
p}$ such that
\begin{numcases}{}
  U \bq^* + \alpha\df(\bx^*) = \mathbf{0},\label{eq:line1}\\
  U \bx^* = \mathbf{0}.\label{eq:line2}
\end{numcases}
\end{lemma}

\begin{proof}
According to Assumption \ref{ass:matrices} and the definition of
$U$, we have
$$
\nul{U}=\nul{V^\T}=\nul{\BW - \W}=\spa{\one}.
$$
Hence from Proposition \ref{conopt}, condition 1, $\bx^*$ is
consensual if and only if \eqref{eq:line2} holds.

Next, following Proposition \ref{conopt}, condition 2, $\bx$ is
optimal if and only if $\one^\T\df(\bx^*)=0$. Since $U$ is
symmetric and $U^\T\one=0$, \eqref{eq:line1} gives
$\one^\T\df(\bx^*)=0$. Conversely, if $\one^\T\df(\bx^*)=0$, then
$\df(\bx^*)\in\spa{U}$ follows from
$\nul{U}=\left(\spa{\one}\right)^\perp$ and thus
$\alpha\df(\bx^*)=-U\bq$ for some $\bq$. Let
$\bq^*=\mathrm{Proj}_U\bq$. Then, $U\bq^*=U\bq$ and
\eqref{eq:line1} holds.\hfill
\end{proof}

Let $\bx^*$ and $\bq^*$ satisfy the optimality conditions
\eqref{eq:line1} and \eqref{eq:line2}. Introduce auxiliary
sequence $$\bq^k = \sum_{t=0}^k U\bx^t$$ and for each  $k$,
\begin{equation} \label{eq:wg}
\bz^k=\left(
             \begin{array}{c}
               \bq^k \\
               \bx^k \\
             \end{array}
           \right),\quad
\bz^*=\left(
             \begin{array}{c}
               \bq^* \\
               \bx^* \\
             \end{array}
           \right),\quad
G=\left(
             \begin{array}{cc}
               I &  \mathbf{0} \\
                \mathbf{0} & \BW \\
             \end{array}
           \right).
\end{equation}
The next lemma establishes the relations among $\bx^k$, $\bq^k$,
$\bx^*$, and $\bq^*$.

\begin{lemma}\label{lemma:recursion} In EXTRA, the quadruple sequence  $\{\bx^k,\bq^k,\bx^*,\bq^*\}$ obeys
\begin{equation} \label{eq:recursion-4}
\begin{array}{rcl}
& &(I + \W - 2\BW) (\bx^{k+1} - \bx^*) + \BW (\bx^{k+1} - \bx^k)\\
&=&- U (\bq^{k+1} - \bq^*) - \alpha [\df(\bx^k) - \df(\bx^*)],
\end{array}
\end{equation}
for any $k=0,1,\cdots$.
\end{lemma}

\begin{proof}
Similar to how \eqref{cw3} is derived, summing EXTRA iterations
$1$ through $k+1$
\begin{equation*} \label{eq:recursion-1}
\begin{array}{cc}
\bx^1 = \W \bx^0 - \alpha \df(\bx^0), \\
\bx^2 = (I + \W) \bx^1 - \BW\bx^0 - \alpha \df(\bx^1) + \alpha \df(\bx^0), \\
\cdots, \\
\bx^{k+1} = (I + \W) \bx^k - \BW \bx^{k-1} - \alpha \df(\bx^k) + \alpha \df(\bx^{k-1}), \\
\end{array}
\end{equation*}
we get
\begin{equation} \label{eq:recursion-2}
\begin{array}{cl}
\bx^{k+1} = \BW \bx^k - \sum\limits_{t=0}^k (\BW - \W) \bx^t -
\alpha \df(\bx^k).
\end{array}
\end{equation}
Using $\bq^{k+1} = \sum_{t=0}^{k+1} U\bx^t$  and the decomposition
$\BW - \W = U^2$, it follows from \eqref{eq:recursion-2} that
\begin{equation} \label{eq:recursion-3}
\begin{array}{cl}
(I + \W - 2\BW) \bx^{k+1} + \BW(\bx^{k+1} - \bx^k) = -U \bq^{k+1}
- \alpha \df(\bx^k).
\end{array}
\end{equation}
Since $(I+\W-2\BW)\one=0$, we have
\begin{equation} \label{eq:line3}
\begin{array}{cl}
(I+\W-2\BW) \bx^*= \mathbf{0}.
\end{array}
\end{equation}
Subtracting \eqref{eq:line3}  from \eqref{eq:recursion-3} and
adding $ \mathbf{0}=U \bq^* + \alpha \df(\bx^*)$ to
\eqref{eq:recursion-3}, we obtain \eqref{eq:recursion-4}. \hfill
\end{proof}

The convergence analysis is based on the recursion
\eqref{eq:recursion-4}. Below we will show that $\bx^k$ converges
to a solution $\bx^*\in \mathcal{X}^*$ and
$\|\bz^{k+1}-\bz^k\|_{\BW}^2$ converges to 0 at a rate of
$O\left(\frac{1}{k}\right)$  in an ergodic sense. Further assuming
(restricted) strong convexity, we obtain the Q-linear convergence
of $\|\bz^k-\bz^*\|_G^2$ to 0, which implies the R-linear
convergence of $\bx^k$ to $\bx^*$.

\subsection{Convergence and Rate}\label{sec:1_k}

Let us first interpret the step size condition
\begin{equation}\label{alphacond}
\alpha <\frac{2\lambda_{\min}(\BW)}{\Lf},
\end{equation} which is assumed by Theorem \ref{lemma:contractive} below. \highlight{First of all, let $W$ satisfy Assumption
\ref{ass:matrices}. It is easy to ensure $\lambda_{\min}(W)\geq 0$
since otherwise, we can replace $W$  by
$\frac{I+W}{2}$. In light of part 4 of Assumption
\ref{ass:matrices}, if we let $\BW=\frac{I+W}{2}$, then we have
$\lambda_{\min}(\BW)\geq\frac{1}{2}$, which {simplifies the
bound \eqref{alphacond} to}
$$\alpha<\frac{1}{\Lf},$$
which is independent of any network property (size, diameter, etc.). Furthermore, if $L_{f_i}$ ($i=1,\ldots,n$) are in the same order,  the bound $\frac{1}{\Lf}$ has the same order as the bound $1/(\frac{1}{n}\sum_{i=1}^n L_{f_i})$, which is used in the (centralized) gradient descent method.  In other words, a fixed and rather large step size is permitted by EXTRA.} 

{\begin{theorem}\label{lemma:contractive} Under
Assumptions  \ref{ass:matrices}--\ref{ass:solotion_existence}, if
$\alpha $ satisfies $0<\alpha<\frac{2\lambda_{\min}(\BW)}{\Lf}$,
then
\begin{equation}\label{eq:contractive}
     \|\bz^k-\bz^*\|_{G}^2-\|\bz^{k+1}-\bz^*\|_{G}^2\geq\zeta\|\bz^k-\bz^{k+1}\|_{G}^2,\quad k=0,1,\ldots,
\end{equation}
where $\zeta=1-\frac{\alpha\Lf}{2\lambda_{\min}(\BW)}$.
\highlight{Furthermore, $\bz^k$ converges to an optimal $\bz^*$.}
\end{theorem}}

\begin{proof}
Following Assumption \ref{ass:functions}, $\df$ is Lipschitz
continuous and thus we have
\begin{equation}\label{eq:convex_proof_1}
\begin{array}{rl}
    &\frac{2\alpha}{\Lf}\|\df(\bx^k)-\df(\bx^*)\|_{\Fro}^2\\
\leq&2\alpha\langle \bx^{k}-\bx^*, \df(\bx^{k}) - \df(\bx^*)\rangle\\
   =&2\langle \bx^{k+1}-\bx^*, \alpha[\df(\bx^{k}) - \df(\bx^*)]\rangle+2\alpha\langle \bx^{k}-\bx^{k+1}, \df(\bx^{k}) - \df(\bx^*)\rangle.
\end{array}
\end{equation}
Substituting \eqref{eq:recursion-4} from Lemma
\ref{lemma:recursion} for $\alpha[\df(\bx^{k}) - \df(\bx^*)]$, it
follows from \eqref{eq:convex_proof_1}
that\begin{equation}\label{eq:convex_proof_2}
 \begin{array}{rcl}
 &    &\frac{2\alpha}{\Lf}\|\df(\bx^k)-\df(\bx^*)\|_{\Fro}^2\\
 &\leq&2\langle \bx^{k+1}-\bx^*,U(\bq^*-\bq^{k+1})\rangle+2\langle \bx^{k+1}-\bx^*,\BW(\bx^k-\bx^{k+1})\rangle\\
 &    &-2\|\bx^{k+1}-\bx^*\|_{I + \W - 2\BW}^2+2\alpha\langle \bx^k-\bx^{k+1}, \df(\bx^k) - \df(\bx^*)\rangle.
 \end{array}
\end{equation}
For the terms on the right-hand-side of \eqref{eq:convex_proof_2},
we have
\begin{equation}\label{eq:temp1_term1}
\begin{array}{rcl}
2\langle \bx^{k+1}-x^*,U(\bq^*-\bq^{k+1})\rangle
&=&2\langle U(\bx^{k+1}-\bx^*),\bq^*-\bq^{k+1}\rangle\\
(\because U\bx^*= \mathbf{0})\quad&=&2\langle U\bx^{k+1},\bq^*-\bq^{k+1}\rangle\\
&=&2\langle \bq^{k+1}-\bq^k,\bq^*-\bq^{k+1}\rangle,\\
\end{array}
\end{equation}
\begin{equation}\label{eq:temp1_term2}
\begin{array}{rcl}
2\langle \bx^{k+1}-\bx^*,\BW(\bx^k-\bx^{k+1})\rangle=2\langle
\bx^{k+1}-\bx^k,\BW(\bx^*-\bx^{k+1})\rangle,
\end{array}
\end{equation}
and
\begin{equation} \label{eq:temp1_term3}
\begin{array}{rcl}
&    &2\alpha\langle \bx^k-\bx^{k+1}, \df(\bx^k) - \df(\bx^*)\rangle\\
&\leq&\frac{\alpha\Lf}{2}\|\bx^k-\bx^{k+1}\|_{\Fro}^2+\frac{2\alpha}{\Lf}\|\df(\bx^k)-\df(\bx^*)\|_{\Fro}^2.
\end{array}
\end{equation}

Plugging \eqref{eq:temp1_term1}--\eqref{eq:temp1_term3} into
\eqref{eq:convex_proof_2} and recalling the definitions of
$\bz^k$, $\bz^*$, and $G$, we have
\begin{equation}\label{eq:convex_proof_3}
 \begin{array}{rcl}
 &    &\frac{2\alpha}{\Lf}\|\df(\bx^k)-\df(\bx^*)\|_{\Fro}^2\\
 &\leq&2\langle \bq^{k+1}-\bq^k,\bq^*-\bq^{k+1}\rangle+ 2\langle \bx^{k+1}-\bx^k,\BW(\bx^*-\bx^{k+1})\rangle\\
 &    &-2\|\bx^{k+1}-\bx^*\|_{I + \W - 2\BW}^2+\frac{\alpha\Lf}{2}\|\bx^k-\bx^{k+1}\|_{\Fro}^2+\frac{2\alpha}{\Lf}\|\df(\bx^k)-\df(\bx^*)\|_{\Fro}^2,
 \end{array}
\end{equation}
that is
\begin{equation}\label{eq:convex_proof_4}
 \begin{array}{l}
 0\leq2\langle \bz^{k+1}-\bz^k,G(\bz^*-\bz^{k+1})\rangle-2\|\bx^{k+1}-\bx^*\|_{I + \W - 2\BW}^2+\frac{\alpha\Lf}{2}\|\bx^k-\bx^{k+1}\|_{\Fro}^2.
 \end{array}
\end{equation}
Apply the basic equality $2\langle
\bz^{k+1}-\bz^k,G(\bz^*-\bz^{k+1})\rangle=\|\bz^k-\bz^*\|_G^2-\|\bz^{k+1}-\bz^*\|_G^2-\|\bz^k-\bz^{k+1}\|_G^2$
to \eqref{eq:convex_proof_4}, we have
\begin{equation}\label{eq:convex_proof_5}
 \begin{array}{rcl}
 0&\leq&\|\bz^k-\bz^*\|_G^2-\|\bz^{k+1}-\bz^*\|_G^2-\|\bz^k-\bz^{k+1}\|_G^2\\
  &    &-2\|\bx^{k+1}-\bx^*\|_{I + \W - 2\BW}^2+\frac{\alpha\Lf}{2}\|\bx^k-\bx^{k+1}\|_{\Fro}^2.
 \end{array}
\end{equation}

Define
$$
G'=\left(
             \begin{array}{cc}
               I &  \mathbf{0} \\
                \mathbf{0} & \BW-\frac{\alpha\Lf}{2}I\\
             \end{array}
           \right).
$$
By Assumption \ref{ass:matrices}, in particular,
$I+\W-2\BW\succcurlyeq0$, we have $\|\bx^{k+1}-\bx^*\|_{I + \W -
2\BW}^2\ge 0$ and thus
$$\|\bz^k-\bz^*\|_G^2-\|\bz^{k+1}-\bz^*\|_G^2\geq\|\bz^k-\bz^{k+1}\|_G - \frac{\alpha\Lf}{2}\|\bx^k-\bx^{k+1}\|_{\Fro}^2= \|\bz^k-\bz^{k+1}\|_{G'}.$$ Since $\alpha<\frac{2\lambda_{\min}(\BW)}{\Lf}$, we have $G'\succ 0$ and
\begin{equation}\label{eq:convex_proof_6}
 \|\bz^k-\bz^{k+1}\|_{G'}^2\geq\zeta\|\bz^k-\bz^{k+1}\|_{G}^2,
\end{equation}
which gives \eqref{eq:contractive}.

\highlight{It shows from \eqref{eq:contractive} that for any optimal $\bz^*$,
$\|\bz^k-\bz^*\|_G^2$ is bounded and contractive, so
$\|\bz^k-\bz^*\|_G^2$ is converging as
$\|\bz^k-\bz^{k+1}\|_G^2\rightarrow0$. The convergence of $\bz^k$
to a solution $\bz^*$ follows from the standard analysis for
contraction methods; see, for example, Theorem 3 in
\cite{He1994}.}\hfill
\end{proof}

To estimate the rate of convergence, we need the following result.
\begin{proposition}\label{lemma:O_1_k}
If a sequence $\{a_k\}\subset\R$ obeys: $a_k\geq0$ and
$\sum_{t=1}^{\infty}a_t<\infty$, then we have\footnote{Part (iii)
is due to \cite{Davis2014}.}:
\emph{(i)}~$\lim_{k\rightarrow\infty}a_k=0$;
\emph{(ii)}~$\frac{1}{k}\sum_{t=1}^{k}a_t=
O\left(\frac{1}{k}\right)$; \emph{(iii)}~$\min_{t\leq
k}\{a_t\}=o\left(\frac{1}{k}\right)$.
\end{proposition}

\begin{proof}
Part (i) is obvious. Let
$b_k\triangleq\frac{1}{k}\sum_{t=1}^{k}a_t$. By the assumptions,
$kb_k$ is uniformly bounded and obeys
$$\lim\limits_{k\rightarrow\infty} kb_k<\infty,$$
from which part (ii) follows. Since
$c_k\triangleq\min\limits_{t\leq k}\{a_t\}$ is monotonically
non-increasing, we have
$$
kc_{2k}=k\times\min\limits_{t\leq
2k}\{a_t\}\leq\sum\limits_{t=k+1}^{2k}a_t.
$$
This and the fact that
$\lim_{k\to\infty}\sum_{t=k+1}^{2k}a_t\rightarrow 0$ give us
$c_{k}=o\left(\frac{1}{k}\right)$ or part (iii).\hfill
\end{proof}

\highlight{\begin{theorem}\label{theorem:1_k}
In the same setting of Theorem \ref{lemma:contractive}, the
following rates hold:

\begin{itemize}
  \item[\emph{(1)}] Running-average progress: $$\frac{1}{k}\sum\limits_{t=1}^{k}\|\bz^t-\bz^{t+1}\|_G^2=
O\left(\frac{1}{k}\right);$$
  \item[\emph{(2)}] Running-best progress: $$\min_{t\leq k}\left\{\|\bz^t-\bz^{t+1}\|_G^2\right\}=o\left(\frac{1}{k}\right);$$
  \item[\emph{(3)}] Running-average  optimality residuals: $$\frac{1}{k}\sum_{t=1}^{k}\|U\bq^t+\alpha\df(\bx^t)\|_{\BW}^2 =O\left(\frac{1}{k}\right) \text{ and } \frac{1}{k}\sum_{t=1}^{k}\|U\bx^t\|_{\Fro}^2 =O\left(\frac{1}{k}\right);$$
  \item[\emph{(4)}] Running-best optimality residuals: $$\min_{t\leq k}\left\{\|U\bq^t+\alpha\df(\bx^t)\|_{\BW}^2\right\}=o\left(\frac{1}{k}\right) \text{ and } \min_{t\leq k}\left\{\|U\bx^t\|_{\Fro}^2\right\} =o\left(\frac{1}{k}\right);$$
\end{itemize}

\end{theorem}

\begin{proof}
{Parts (1) and  (2): Since the individual terms
$\|\bz^k-\bz^*\|_G^2$ converge to $0$, we are able to sum
\eqref{eq:contractive} in Theorem \ref{lemma:contractive} over $k=0$
through $\infty$ and apply the telescopic cancellation, i.e.,
\begin{equation}\label{eq:sum_bounded}
 \begin{array}{c}
 \sum\limits_{t=0}^{\infty}\|\bz^t-\bz^{t+1}\|_G^2=\frac{1}{\delta}\sum\limits_{t=0}^{\infty}\left(\|\bz^t-\bz^*\|_G^2-\|\bz^{t+1}-\bz^*\|_G^2\right)=\frac{\|\bz_0-\bz^*\|_G^2}{\delta}<\infty.
 \end{array}
\end{equation}
Then, the results follow from Proposition \ref{lemma:O_1_k}
immediately.

Parts (3) and (4): The progress $\|\bz^k-\bz^{k+1}\|_G^2$ can be
interpreted as the residual to the first-order optimality
condition. In light of the first-order optimality conditions
\eqref{eq:line1} and \eqref{eq:line2} in Lemma \ref{lemma:opt},
the \emph{optimality residuals} are defined as
$\|U\bq^k+\alpha\df(\bx^k)\|_{\BW}^2$ and $\|U\bx^k\|_{\Fro}^2$.
Furthermore, $\|\frac{1}{\alpha}\one^\T(U\bq^k+\alpha\df(\bx^k))\|_2^2=\|\nabla
f_1(x_{(1)}^k)+...+\nabla f_n(x_{(n)}^k)\|_2^2$ is the violation to
the first-order optimality of \eqref{eq:F}, while
$\|U\bx^k\|_{\Fro}^2$ is the violation of consensus. Below we
obtain the convergence rates of the optimality residuals.

Using the basic inequality
$\|\mathbf{a}+\mathbf{b}\|_{\Fro}^2\geq\frac{1}{\rho}\|\mathbf{a}\|_{\Fro}^2-\frac{1}{\rho-1}\|\mathbf{b}\|_{\Fro}^2$
which holds for any $\rho>1$ and any matrices $\mathbf{a}$ and
$\mathbf{b}$ of the same size, it follows that
\begin{equation}\label{eq:FOOR}
 \begin{array}{rcl}
 \|\bz^k-\bz^{k+1}\|_G^2&=&\|\bq^k-\bq^{k+1}\|_{\Fro}^2+\|\bx^k-\bx^{k+1}\|_{\BW}^2\\
                        &=&\|\bx^{k+1}\|_{\BW-W}^2+\|(I-\BW)\bx^k+U\bq^k+\alpha\df(\bx^k)\|_{\BW}^2\\
                        &\geq&\|\bx^{k+1}\|_{\BW-W}^2+\frac{1}{\rho}\|U\bq^k+\alpha\df(\bx^k)\|_{\BW}^2-\frac{1}{\rho-1}\|(I-\BW)\bx^k\|_{\BW}^2.
 \end{array}
\end{equation}

Since $\BW-W$ and $(I-\BW)\BW(I-\BW)$ are symmetric and
$$\nul{\BW-W}\subseteq\nul{(I-\BW)\BW(I-\BW)},$$ there exists a bounded $\upsilon>0$ such that
$\|(I-\BW)\bx^k\|_{\BW}^2=\|\bx^k\|_{(I-\BW)\BW(I-\BW)}^2\leq\upsilon\|\bx^k\|_{\BW-W}^2$.
It follows from \eqref{eq:FOOR} that
\begin{equation}\label{eq:FOOR_2}
 \begin{array}{rcl}
 &    &\frac{1}{k}\sum\limits_{t=1}^{k}\|\bz^t-\bz^{t+1}\|_G^2+\frac{1}{k}\|\bx^{1}\|_{\BW-W}^2\\
 &\geq&\frac{1}{k}\sum\limits_{t=1}^{k}\left(\|\bx^{t+1}\|_{\BW-W}^2-\frac{\upsilon}{\rho-1}\|\bx^t\|_{\BW-W}^2\right)+\frac{1}{k}\|\bx^{1}\|_{\BW-W}^2\\
 &    &+\frac{1}{k}\sum\limits_{t=1}^{k}\frac{1}{\rho}\|U\bq^t+\alpha\df(\bx^t)\|_{\BW}^2\text{ (Set $\rho>\upsilon+1$)}\\
 & =  &\frac{1}{k}\sum\limits_{t=1}^{k}(1-\frac{\upsilon}{\rho-1})\|U\bx^t\|_{\Fro}^2
 +\frac{1}{k}\|\bx^{k+1}\|_{\BW-W}^2+\frac{1}{k}\sum\limits_{t=1}^{k}\frac{1}{\rho}\|U\bq^t+\alpha\df(\bx^t)\|_{\BW}^2.
 \end{array}
\end{equation}

As part (1) shows that
$\frac{1}{k}\sum_{t=1}^{k}\|\bz^t-\bz^{t+1}\|_G^2=O\left(\frac{1}{k}\right)$,
we have
$\frac{1}{k}\sum_{t=1}^{k}\|U\bq^t+\alpha\df(\bx^t)\|_{\BW}^2 =
O\left(\frac{1}{k}\right)$ and
$\frac{1}{k}\sum_{t=1}^{k}\|U\bx^t\|_{\Fro}^2 =
O\left(\frac{1}{k}\right)$.

From \eqref{eq:FOOR_2} and \eqref{eq:sum_bounded}, we see that
both $\|U\bq^t+\alpha\df(\bx^t)\|_{\BW}^2$ and
$\|U\bx^t\|_{\Fro}^2$ are summable. Again, by Proposition
\ref{lemma:O_1_k}, we have part (4), the
$o\left(\frac{1}{k}\right)$ rate of running best first-order
optimality residuals.} \hfill
\end{proof}}

It is open whether  $\|\bz^k-\bz^{k+1}\|_G^2$ is monotonic or not. If one can show its monotonicity, then  the convergence rates will hold for the last point in the running sequence. 


\subsection{Linear Convergence under Restricted Strong Convexity}
In this subsection we prove that EXTRA with a proper step size
reaches linear convergence if the original objective $\bar{f}$ is
restricted strongly convex.

A convex function $h: \R^p\rightarrow \R$ is \emph{strongly
convex} if there exists $\mu>0$ such that
$$\langle \nabla h(x_a)-\nabla h(x_b),x_a-x_b \rangle \ge \mu \|x_a - x_b\|^2,\quad \forall x_a,x_b\in\R^p.$$
 $h$  is \emph{restricted strongly convex}\footnote{There are different definitions of restricted strong convexity. Ours is derived from \cite{Lai2013}.} with respect to point $\tilde{x}$ if there exists  $\mu>0$ such that $$\langle\nabla h(x)-\nabla h(\tilde{x}),x-\tilde{x}\rangle\geq\mu\|x-\tilde{x}\|_2^2,\quad \forall x\in\R^p.$$

For proof convenience, we introduce  function
$$\g(\bx)\triangleq\f(\bx)+\frac{1}{4\alpha}\|\bx\|_{\BW-W}^2$$
and claim that $\bar{f}$ is restricted strongly convex with
respect to its solution $x^*$ if, and only if, $\g$ is so with
respect to $\bx^* = \one (x^*)^\T$.

\begin{proposition}\label{prop:strongly_convex} Under Assumptions \ref{ass:matrices} and \ref{ass:functions}, the following two statements are equivalent:
\begin{enumerate}
  \item[(i)] The original objective $\bar{f}(x)=\frac{1}{n}\sum_{i=1}^n f_i(x)$ is restricted strongly convex with respect to $x^*$;
  \item[(ii)] The penalized function $\g(\bx)=\f(\bx)+\frac{1}{4\alpha}\|\bx\|_{\BW-W}^2$ is restricted strongly convex with respect to $\bx^*$.
\end{enumerate}
\highlight{In addition, the strong convexity constant of $\g$ is no less than that  of $\bar{f}$.}
\end{proposition}

See Appendix \ref{sec:str_cvx_proof} for its proof.

\begin{theorem} \label{theorem:linear}
If $\g(\bx)\triangleq\f(\bx)+\frac{1}{4\alpha}\|\bx\|_{\BW-\W}^2$
is restricted strongly convex with respect to $\bx^*$ with
constant $\mug>0$, then with proper step size
$\alpha<\frac{2\mug\lambda_{\min}(\BW)}{\Lf^2}$, there exists
$\delta>0$ such that the sequence $\{\bz^k\}$ generated by EXTRA
satisfies
\begin{equation} \label{eq:main}
\begin{array}{cl}
\|\bz^k - \bz^*\|_G^2\geq(1+\delta)\|\bz^{k+1} - \bz^*\|_G^2.
\end{array}
\end{equation}
That is, $\|\bz^k-\bz^*\|_G^2$ converges to $0$ at the Q-linear
rate $O\big((1+\delta)^{-k}\big)$. Consequently,
$\|\bx^k-\bx^*\|_{\BW}^2$ converges to $0$ at the R-linear rate
$O\big((1+\delta)^{-k}\big)$.
\end{theorem}

\begin{proof}
{\textbf{Toward a lower bound of
$\|\bz^k-\bz^*\|_G^2-\|\bz^{k+1}-\bz^*\|_G^2$:}} From the
definition of $\g$ and its restricted strong convexity, we have
\begin{equation}\label{eq:linear_proof_1}
 \begin{array}{rcl}
2\alpha \mug\|\bx^{k+1}-\bx^*\|_{\Fro}^2 &\leq&2\alpha\langle \bx^{k+1}-\bx^*,\dg(\bx^{k+1}) - \dg(\bx^*)\rangle\\
 & =  &\|\bx^{k+1}-\bx^*\|_{\BW-\W}^2+2\alpha\langle \bx^{k+1}-\bx^*,\df(\bx^{k+1}) - \df(\bx^k)\rangle\\
 &    &+2\langle \bx^{k+1}-\bx^*,\alpha[\df(\bx^{k}) - \df(\bx^*)]\rangle.
 \end{array}
\end{equation}
Using Lemma \ref{lemma:recursion} for $\alpha[\df(\bx^{k}) -
\df(\bx^*)]$ in  \eqref{eq:linear_proof_1}, we get
\begin{equation}\label{eq:linear_proof_2}
 \begin{array}{rcl}
 &    &2\alpha \mug\|\bx^{k+1}-\bx^*\|_{\Fro}^2\\
 &\leq&\|\bx^{k+1}-\bx^*\|_{\BW-\W}^2+2\alpha\langle \bx^{k+1}-\bx^*,\df(\bx^{k+1}) - \df(\bx^k)\rangle-2\|\bx^{k+1}-\bx^*\|_{I + \W - 2\BW}^2\\
 &    &+2\langle \bx^{k+1}-\bx^*,U(\bq^*-\bq^{k+1})\rangle+2\langle \bx^{k+1}-\bx^*,\BW(\bx^k-\bx^{k+1})\rangle\\
 & =  &\|\bx^{k+1}-\bx^*\|_{(\BW-\W)-2(I + \W - 2\BW)}^2+2\alpha\langle \bx^{k+1}-\bx^*,\df(\bx^{k+1}) - \df(\bx^k)\rangle\\
 &    &+2\langle \bx^{k+1}-\bx^*,U(\bq^*-\bq^{k+1})\rangle+2\langle \bx^{k+1}-\bx^*,\BW(\bx^k-\bx^{k+1})\rangle.
 \end{array}
\end{equation}

For the last three terms on the right-hand side of
\eqref{eq:linear_proof_2}, we have from Young's inequality
\begin{equation} \label{eq:temp2_term1}
\begin{array}{rcl}
&    &2\alpha\langle \bx^{k+1}-\bx^*, \df(\bx^{k+1}) - \df(\bx^k)\rangle\\
&\leq&\alpha\eta\|\bx^{k+1}-\bx^*\|_{\Fro}^2+\frac{\alpha}{\eta}\|\df(\bx^{k+1})-\df(\bx^k)\|_{\Fro}^2\\
&\leq&\alpha\eta\|\bx^{k+1}-\bx^*\|_{\Fro}^2+\frac{\alpha
\Lf^2}{\eta}\|\bx^{k+1} - \bx^k\|_{\Fro}^2,
\end{array}
\end{equation}
where $\eta>0$ is a tunable parameter and
\begin{equation}\label{eq:temp2_term2}
\begin{array}{c}
2\langle \bx^{k+1}-\bx^*,U(\bq^*-\bq^{k+1})\rangle=2\langle
\bq^{k+1}-\bq^k,\bq^*-\bq^{k+1}\rangle,
\end{array}
\end{equation}
and
\begin{equation}\label{eq:temp2_term3}
\begin{array}{rcl}
2\langle \bx^{k+1}-\bx^*,\BW(\bx^k-\bx^{k+1})\rangle=2\langle
\bx^{k+1}-\bx^k,\BW(\bx^*-\bx^{k+1})\rangle.
\end{array}
\end{equation}
Plugging \eqref{eq:temp2_term1}--\eqref{eq:temp2_term3} into
\eqref{eq:linear_proof_2}  and recalling the definition of
$\bz^k$, $\bz^*$, and $G$, we obtain
\begin{equation}\label{eq:linear_proof_3}
 \begin{array}{rcl}
&    &2\alpha \mug\|\bx^{k+1}-\bx^*\|_{\Fro}^2\\ &\leq&\|\bx^{k+1}-\bx^*\|_{(\BW-\W)-2(I + \W - 2\BW)}^2+\alpha\eta\|\bx^{k+1}-\bx^*\|_{\Fro}^2\\
 &    &+\frac{\alpha \Lf^2}{\eta}\|\bx^{k+1} - \bx^k\|_{\Fro}^2+2\langle \bz^{k+1}-\bz^k,G(\bz^*-\bz^{k+1})\rangle.
 \end{array}
\end{equation}
By $2\langle
\bz^{k+1}-\bz^k,G(\bz^*-\bz^{k+1})=\|\bz^k-\bz^*\|_G^2-\|\bz^{k+1}-\bz^*\|_G^2-\|\bz^k-\bz^{k+1}\|_G^2$,
\eqref{eq:linear_proof_3} turns into
\begin{equation}\label{eq:linear_proof_4}
 \begin{array}{rcl}
\|\bz^k-\bz^*\|_G^2-\|\bz^{k+1}-\bz^*\|_G^2
&\geq&\|\bx^{k+1}-\bx^*\|_{\alpha(2\mug-\eta)I-(\BW-\W)+2(I+\W-2\BW)}^2\\
&    &+\|\bz^k-\bz^{k+1}\|_G^2-\frac{\alpha \Lf^2}{\eta}\|\bx^k -
\bx^{k+1}\|_{\Fro}^2.
\end{array}
\end{equation}

\vspace{1em}

\noindent{\textbf{A critical inequality:}} In order to
establish \eqref{eq:main}, in light of \eqref{eq:linear_proof_4},
it remains to show
\begin{equation}\label{eq:linear_proof_5}
\begin{array}{rcl} & &\|\bx^{k+1}-\bx^*\|_{\alpha(2\mug-\eta)I-(\BW-\W)+2(I+\W-2\BW)}^2+\|\bz^k-\bz^{k+1}\|_G^2-\frac{\alpha \Lf^2}{\eta}\|\bx^k - \bx^{k+1}\|_{\Fro}^2\\
&\geq&\delta\|\bz^{k+1}-\bz^*\|_G^2.
\end{array}
\end{equation}
With the terms of $\bz^k$, $\bz^*$ in \eqref{eq:linear_proof_5}
expanded and from
$\|\bx^{k+1}-\bx^*\|_{\BW-\W}^2=\|U(\bx^{k+1}-\bx^*)\|_{\Fro}^2=\|U\bx^{k+1}\|_{\Fro}^2=\|\bq^{k+1}-\bq^k\|_{\Fro}^2$,
\eqref{eq:linear_proof_5} is equivalent to
\begin{equation}\label{eq:linear_proof_6}
 \begin{array}{c}
 \|\bx^{k+1}-\bx^*\|_{\alpha(2\mug-\eta)I+2(I+\W-2\BW)-\delta\BW}^2+\|\bx^k-\bx^{k+1}\|_{\BW-\frac{\alpha \Lf^2}{\eta}I}^2\geq\delta\|\bq^{k+1}-\bq^*\|_{\Fro}^2,
 \end{array}
\end{equation}
which is \emph{what remains to be shown below}. That is, we must
find a upper bound for $\|\bq^{k+1}-\bq^*\|_{\Fro}^2$ in terms of
$\|\bx^{k+1}-\bx^*\|_{\Fro}^2$ and $\|\bx^k-\bx^{k+1}\|_{\Fro}^2$.
\vspace{1em}

\noindent\textbf{Establishing \eqref{eq:linear_proof_6}, Step 1:} From
Lemma \ref{lemma:recursion} we have
\begin{equation} \label{eq:linear_proof_7}
\begin{array}{rcl}
& &\|U(\bq^{k+1} - \bq^*)\|_{\Fro}^2\\
&=&\|(I + \W - 2\BW) (\bx^{k+1} - \bx^*) + \BW (\bx^{k+1} - \bx^k)+\alpha [\df(\bx^k)-\df(\bx^*)]\|_{\Fro}^2\\
&=&\|(I + \W - 2\BW) (\bx^{k+1} - \bx^*)+\alpha [\df(\bx^{k+1})-\df(\bx^*)]\\
& &+ \BW(\bx^{k+1} - \bx^k)+\alpha
[\df(\bx^k)-\df(\bx^{k+1})]\|_{\Fro}^2.
\end{array}
\end{equation}
From the inequality
$\|\mathbf{a}+\mathbf{b}+\mathbf{c}+\mathbf{d}\|_{\Fro}^2\leq\theta\left(\frac{\beta}{\beta-1}\|\mathbf{a}\|_{\Fro}^2+\beta\|\mathbf{b}\|_{\Fro}^2\right)+\frac{\theta}{\theta-1}\left(\frac{\gamma}{\gamma-1}\|\mathbf{c}\|_{\Fro}^2+\gamma\|\mathbf{d}\|_{\Fro}^2\right)$,
which holds for any $\theta>1$, $\beta>1$, $\gamma>1$ and any
matrices $\mathbf{a}$, $\mathbf{b}$, $\mathbf{c}$, $\mathbf{d}$ of
the same dimensions, it follows that
\begin{equation}\label{eq:linear_proof_8}
\begin{array}{rcl}
&    &\|\bq^{k+1}-\bq^*\|_{\BW-\W}^2\\
&\leq&\theta\left(\frac{\beta}{\beta-1}\|\bx^{k+1}-\bx^*\|_{(I+\W-2\BW)^2}^2+\beta \alpha^2\|\df(\bx^{k+1})-\df(\bx^*)\|_{\Fro}^2\right)\\
&    &
+\frac{\theta}{\theta-1}\left(\frac{\gamma}{\gamma-1}\|\bx^k-\bx^{k+1}\|_{\BW^2}^2+\gamma
\alpha^2\|\df(\bx^k)-\df(\bx^{k+1})\|_{\Fro}^2\right).
\end{array}
\end{equation}
By Lemma \ref{lemma:opt} and the definition of $\bq^k$, all the
columns of  $\bq^*$ and $\bq^{k+1}$ lie in the column space of
$\BW-\W$. This together with the Lipschitz continuity of
$\df(\bx)$ turns \eqref{eq:linear_proof_8}  into
\begin{equation}\label{eq:linear_proof_9}
\begin{array}{rcl}
&    &\|\bq^{k+1}-\bq^*\|_{\Fro}^2\\
&\leq&\frac{\theta}{\tilde\lambda_{\min}(\BW-\W)}\left(\frac{\beta\lambda_{\max}((I+\W-2\BW)^2)}{\beta-1}+\beta \alpha^2\Lf^2\right)\|\bx^{k+1}-\bx^*\|_{\Fro}^2\\
&    &
+\frac{\theta}{(\theta-1)\tilde\lambda_{\min}(\BW-\W)}\left(\frac{\gamma\lambda_{\max}(\BW^2)}{\gamma-1}+\gamma
\alpha^2\Lf^2\right)\|\bx^k-\bx^{k+1}\|_{\Fro}^2,
\end{array}
\end{equation}
where $\tilde{\lambda}_{\min}(\cdot)$ gives the smallest
\emph{nonzero} eigenvalue. To make a rather tight bound, we choose
$\gamma=1+\frac{\sigma_{\max}(\BW)}{\alpha\Lf}$ and
$\beta=1+\frac{\sigma_{\max}(I+\W-2\BW)}{\alpha\Lf}$ in
\eqref{eq:linear_proof_9} and obtain
\begin{equation}\label{eq:linear_proof_10}
\begin{array}{rl}
    &\|\bq^{k+1}-\bq^*\|_{\Fro}^2\\
\leq&\frac{\theta(\sigma_{\max}(I+\W-2\BW)+\alpha
\Lf)^2}{\tilde\lambda_{\min}(\BW-\W)}\|\bx^{k+1}-\bx^*\|_{\Fro}^2+\frac{\theta(\sigma_{\max}(\BW)+\alpha
\Lf)^2}{(\theta-1)\tilde\lambda_{\min}(\BW-\W)}\|\bx^k-\bx^{k+1}\|_{\Fro}^2.
\end{array}
\end{equation}

\noindent\textbf{Establishing \eqref{eq:linear_proof_6}, Step 2:} In
order to establish \eqref{eq:linear_proof_6}, with
\eqref{eq:linear_proof_10}, it only remains to show
\begin{equation}\label{eq:linear_proof_11}
\begin{array}{rcl}
&&
\|\bx^{k+1}-\bx^*\|_{\alpha(2\mug-\eta)I+2(I+\W-2\BW)-\delta\BW}^2+\|\bx^k-\bx^{k+1}\|_{\BW-\frac{\alpha \Lf^2}{\eta}I}^2\\
 &\geq&\delta\left(\frac{\theta(\sigma_{\max}(I+\W-2\BW)+\alpha \Lf)^2}{\tilde\lambda_{\min}(\BW-\W)}\|\bx^{k+1}-\bx^*\|_{\Fro}^2+\frac{\theta(\sigma_{\max}(\BW)+\alpha \Lf)^2}{(\theta-1)\tilde\lambda_{\min}(\BW-\W)}\|\bx^k-\bx^{k+1}\|_{\Fro}^2\right).
\end{array}
\end{equation}
To validate \eqref{eq:linear_proof_11}, we need
\begin{equation}\label{eq:linear_proof_12}
\begin{array}{cl}
\left\{
  \begin{array}{c}
  \alpha(2\mug-\eta)I+2(I+\W-2\BW)-\delta\BW\succcurlyeq\frac{\delta\theta(\sigma_{\max}(I+\W-2\BW)+\alpha \Lf)^2}{\tilde\lambda_{\min}(\BW-\W)}I,\\
  \BW-\frac{\alpha \Lf^2}{\eta}I\succcurlyeq\frac{\delta\theta(\sigma_{\max}(\BW)+\alpha \Lf)^2}{(\theta-1)\tilde\lambda_{\min}(\BW-\W)}I,\\
  \end{array}
\right.
\end{array}
\end{equation}
which holds as long as
\begin{equation}\label{eq:explicit_rate}
\begin{array}{c}
\delta\leq\min\left\{\frac{\alpha(2\mug-\eta)\tilde\lambda_{\min}(\BW-\W)}{\theta(\sigma_{\max}(I+\W-2\BW)+\alpha
\Lf)^2+\lambda_{\max}(\BW)\tilde\lambda_{\min}(\BW-\W)},\frac{(\theta-1)(\eta\lambda_{\min}(\BW)-\alpha
\Lf^2)\tilde\lambda_{\min}(\BW-\W)}{\theta\eta(\sigma_{\max}(\BW)+\alpha
\Lf)^2} \right\}
\end{array}.
\end{equation}
To ensure $\delta>0$, the following conditions are what we finally
need:
\begin{equation}\label{eq:additional_cond2}
\begin{array}{c}
\eta\in\left(0,2\mug\right)\ \ \mathrm{and}\ \
\alpha\in\left(0,\frac{\eta\lambda_{\min}(\BW)}{\Lf^2}\right)\triangleq\mathcal{S}.
\end{array}
\end{equation}
Obviously set $\mathcal{S}$ is nonempty. Therefore, with a proper
step size $\alpha\in\mathcal{S}$, the sequences
$\|\bz^k-\bz^*\|_G^2$ is Q-linearly convergent to $0$ at the rate
$O\big((1+\delta)^{-k}\big)$. Since the definition of $G$-norm
implies $\|\bx^k-\bx^*\|_{\BW}^2\leq\|\bz^k-\bz^*\|_G^2$,
$\|\bx^k-\bx^*\|_{\BW}^2$ is R-linearly convergent to $0$ at the
same rate.\hfill
\end{proof}

\begin{remark}[Strong convexity condition for linear convergence]\label{remark:BW} The restricted strong convexity assumption in Theorem
\ref{theorem:linear} is imposed on
$\g(\bx)=\f(\bx)+\frac{1}{4\alpha}\|\bx\|_{\BW-\W}^2$, not on
$\f(\bx)$. In other words, the linear convergence of EXTRA  does
not require all $f_i$ to be individually (restricted) strongly
convex.
\end{remark}

\begin{remark}[Acceleration by overshooting $\BW$]
For conciseness, we used Assumption \ref{ass:matrices} for both
Theorems \ref{theorem:1_k} and \ref{theorem:linear}. In fact, for
Theorem \ref{theorem:linear}, the condition
$\frac{I+W}{2}\succcurlyeq \BW$
 in part 4 of Assumption \ref{ass:matrices} can be relaxed, thanks to $\mug$ in
\eqref{eq:linear_proof_12}. Certain
$\BW\succcurlyeq\frac{I+W}{2}$, such as $\BW=\frac{1.5I+W}{2.5}$,
can still give linear convergence. In fact, we observed that such
an ``overshot'' choice of $\BW$ can slightly accelerate the
convergence of EXTRA.
\end{remark}

\highlight{\begin{remark}[Step size optimization]We tried
deriving an optimal step size and corresponding explicit linear
convergence rate by optimizing certain quantities that appear in
the proof, but it becomes quite tricky and messy. For the special
case $\BW=\frac{I+\W}{2}$, by taking $\eta\rightarrow \mug$, we
get {a satisfactory} step size
$\alpha\rightarrow\frac{\mug(1+\lambda_{\min}(W))}{4\Lf^2}$.
\end{remark}}
%
\begin{remark}[Step size for ensuring linear convergence] Interestingly, the critical step size, $\alpha<\frac{2\mug\lambda_{\min}(\BW)}{\Lf^2}=O\left(\frac{\mug}{\Lf^2}\right)$, in \eqref{eq:additional_cond2} for ensuring the linear convergence, and the parameter $\alpha=\frac{\tilde\lambda_{\min}(\BW-W)}{2(1+\frac{1}{\gamma^2})(\mubarf-2\Lf\gamma)}=O\left(\frac{\mug}{\Lf^2}\right)$ in \eqref{eq:ass_equivalence_proof_2_5} for ensuring the restricted strong convexity with $O(\mug)=O(\mubarf)$, have the same order.

On the other hand, we numerically observed that a step size as large as $O\left(\frac{1}{\Lf}\right)$ still leads to linear convergence, and EXTRA becomes faster with this larger step size. It remains an open question to prove linear convergence under this larger step size. 
\end{remark}

\subsection{Decentralized implementation}

\highlight{We shall explain how to perform EXTRA with only local computation
and neighbor communication. EXTRA's formula is formed by
$\df(\bx)$, $\W\bx$ and $\BW\bx$, and $\alpha$. By definition
$\df(\bx)$ is local computation. Assumption \ref{ass:matrices}
part 1 ensures that $\W\bx$ and $\BW\bx$ can be computed with
local and neighbor information. Following our convergence theorems
above, determining $\alpha$ requires the bounds on $\Lf$ and
$\lambda_{\min}(\BW)$, as well as that of $\mug$ in the
(restricted) strongly convex case. As we have argued at the
beginning of Subsection \ref{sec:1_k}, it is easy to ensure
$\lambda_{\min}(\BW)\ge \frac{1}{2}$, so $\lambda_{\min}(\BW)$ can
be conservatively set as $\frac{1}{2}$. To obtain
$\Lf=\max_i\{L_{f_i}\}$, a maximum consensus algorithm is needed.
On the other hand, it is tricky to determine $\mug$ or its lower bound
$\mu_{\bar{f}}$, except in the case that each $f_i$ is (restricted)
strongly convex, we can conservatively use $\min_i\{\mu_{f_i}\}$.
When no bound $\mug$ is available in the (restricted) strongly
convex case, setting $\alpha$ according to the general convex case
(subsection \ref{sec:1_k}) often still leads to linear
convergence.}



\section{Numerical Experiments}\label{sc:num}

\subsection{Decentralized Least Squares}\label{sec:DLS}
Consider a decentralized sensing problem: each agent
$i\in\{1,\cdots,n\}$ holds its own measurement equation,
$y_{(i)}=M_{(i)} x+e_{(i)}$, where $y_{(i)}\in\R^{m_i}$ and
$M_{(i)}\in\R^{m_i\times p}$ are measured data, $x\in\mathbb{R}^p$
is unknown signal, and $e_{(i)}\in\mathbb{R}^{m_i}$ is unknown
noise. The goal is to estimate $x$. We apply the least squares
loss and try to solve
$$
\Min\limits_x \bar{f}(x)=\frac{1}{n}\sum\limits_{i=1}^n \frac{1}{2}\|M_{(i)} x-y_{(i)}\|_2^2.
$$
The network in this experiment is randomly generated with
connectivity ratio $r=0.5$, where $r$ is defined as the number of
edges divided by $\frac{L(L-1)}{2}$, the number of all possible
ones. We set $n=10$, $m_i=1, \forall i$, $p=5$. Data $y_{(i)}$ and $M_{(i)}$,
as well as noise $e_{(i)}$, $\forall i$, are generated following
the standard normal distribution. We normalize the data so that
$\Lf=1$. The algorithm starts from $x_{(i)}^{0}=0, \forall i$, and
$\|x^*-x_{(i)}^{0}\|=300$.

We use the same matrix $W$ by strategy (iv) in Section
\ref{sec:matrices} for both DGD and EXTRA. For EXTRA, we simply
use the aforementioned matrix $\BW=\frac{I+\W}{2}$. We run DGD
with a fixed step size $\alpha$, a diminishing one
$\frac{\alpha}{k^{1/3}}$ \cite{Jakovetic2014}, a diminishing one
$\frac{\alpha^0}{k^{1/3}}$ with hand-optimized $\alpha^0$, a
diminishing one $\frac{\alpha}{k^{1/2}}$ \cite{Chen2012}, and a
diminishing one $\frac{\alpha^0}{k^{1/2}}$ with hand-optimized
$\alpha^0$, where $\alpha$ is the theoretical critical step size
given in \cite{Kun2014}. We let EXTRA use the same fixed step size
$\alpha$.

The numerical results are illustrated in Fig. \ref{eps:num_LS}. In
this experiment, we observe that both DGD with the fixed step size
and EXTRA show similar linear convergence in the first $200$
iterations. Then DGD with the fixed step size begins to slow down
and eventually stall, and EXTRA continues its progress.
\begin{figure}[H]
\begin{center}
\includegraphics[height=4.8cm]{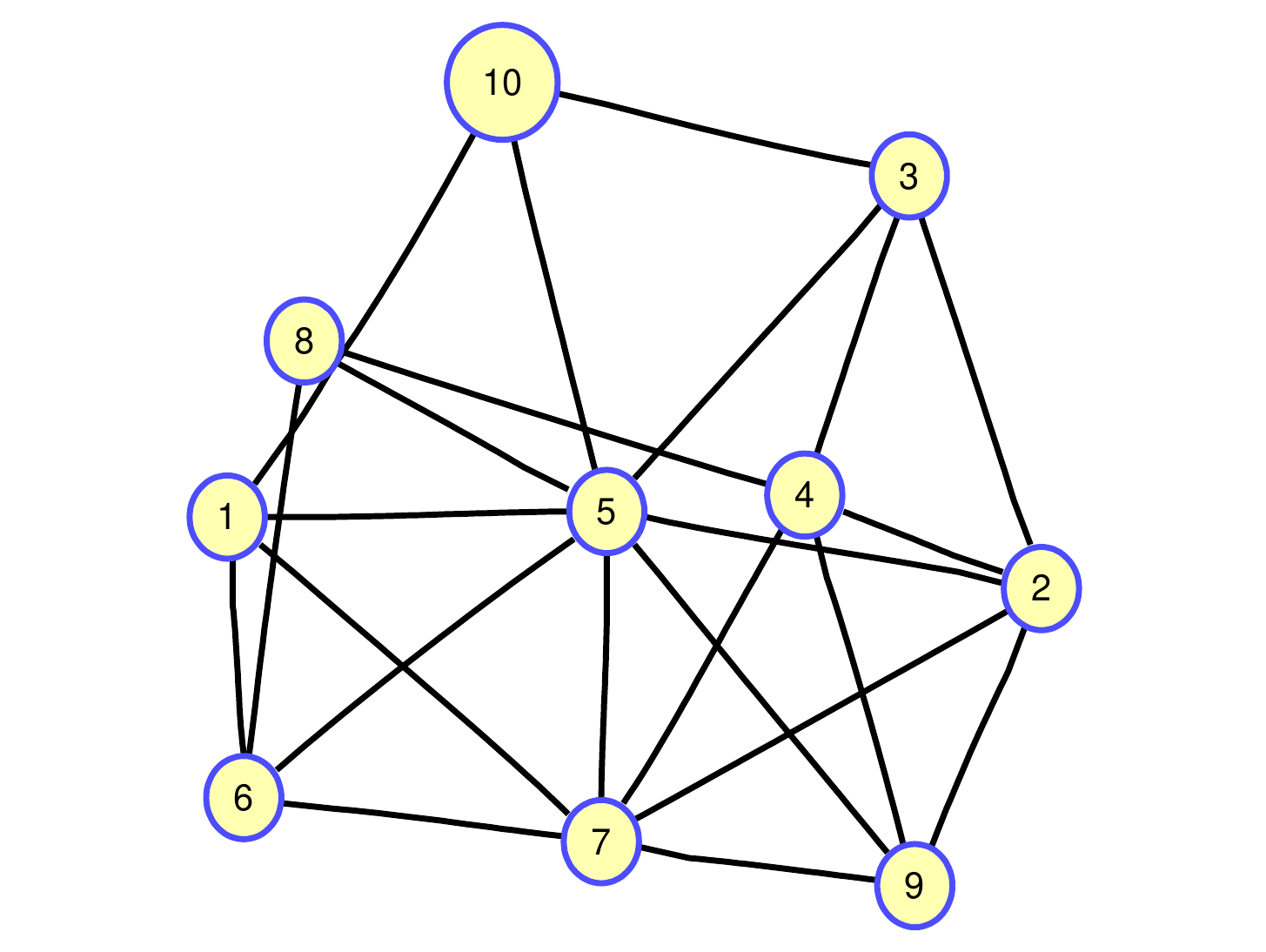}
\includegraphics[height=4.8cm]{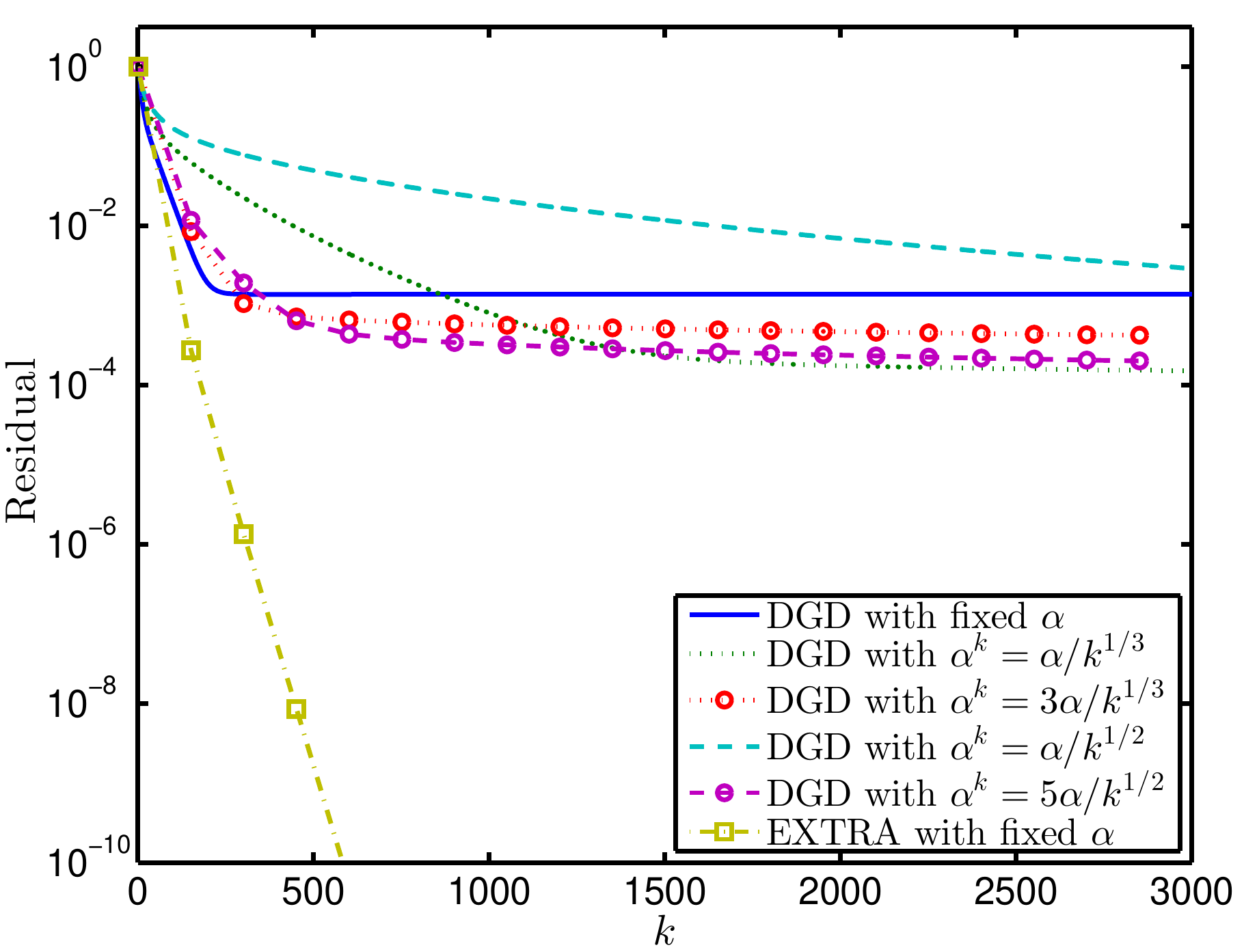}
\caption{Plot of residuals $\frac{\|\bx^k-\bx^*\|_\Fro}{\|\bx^0-\bx^*\|_\Fro}$. Constant $\alpha=0.5276$ is the theoretical critical step size given for DGD in \cite{Kun2014}. For DGD with diminishing step sizes $O(1/k^{1/3})$ and $O(1/k^{1/2})$, we have \emph{hand-optimized} their initial step sizes as $3\alpha$ and $5\alpha$, respectively.
}\label{eps:num_LS}
\end{center}
\end{figure}

\subsection{Decentralized Robust Least Squares}\label{sec:DRLS}
Consider the same decentralized sensing setting and network as in Section \ref{sec:DLS}. In this experiment, we use the Huber loss, which is known to be robust to outliers, and it allows us to observe both sublinear and linear convergence. We call the problem as decentralized robust least squares:
$$
\Min\limits_{x} \bar{f}(x)=\frac{1}{n}\sum\limits_{i=1}^n\left\{\sum\limits_{j=1}^{m_i}H_\xi(M_{(i)j} x-y_{(i)j})\right\},
$$
where $M_{(i)j}$ is the $j$-th row of matrix $M_{(i)}$ and $y_{(i)j}$ is the $j$-th entry of vector $y_{(i)}$. The Huber loss
function $H_\xi$ is defined as
$$H_\xi(a)=\left\{
           \begin{array}{cl}
             \frac{1}{2}a^2, &\text{ for $|a|\leq\xi$,\quad ($\ell_2^2$ zone)},\\
             \xi(|a|-\frac{1}{2}\xi), &\text{\hspace{2pt} otherwise,\hspace{8pt} ($\ell_1$ zone)}.\\
           \end{array}
         \right.$$
We set $\xi=2$. The optimal solution $x^*$ is artificially set in the
$\ell_2^2$ zone while $x^0_{(i)}$ is set in the $\ell_1$ zone at all agents $i$.

Except for new hand-optimized initial step sizes for DGD's diminishing step sizes, all other algorithmic parameters remain unchanged from the last test.

The numerical results are illustrated in Fig. \ref{eps:num_Huber}. EXTRA has sublinear convergence for the fist 1000 iterations and then begins linear convergence, {as $x_{(i)}^k$ for most $i$ enter the $\ell_2^2$ zone.}
\begin{figure}[H]
\begin{center}
\includegraphics[height=4.8cm]{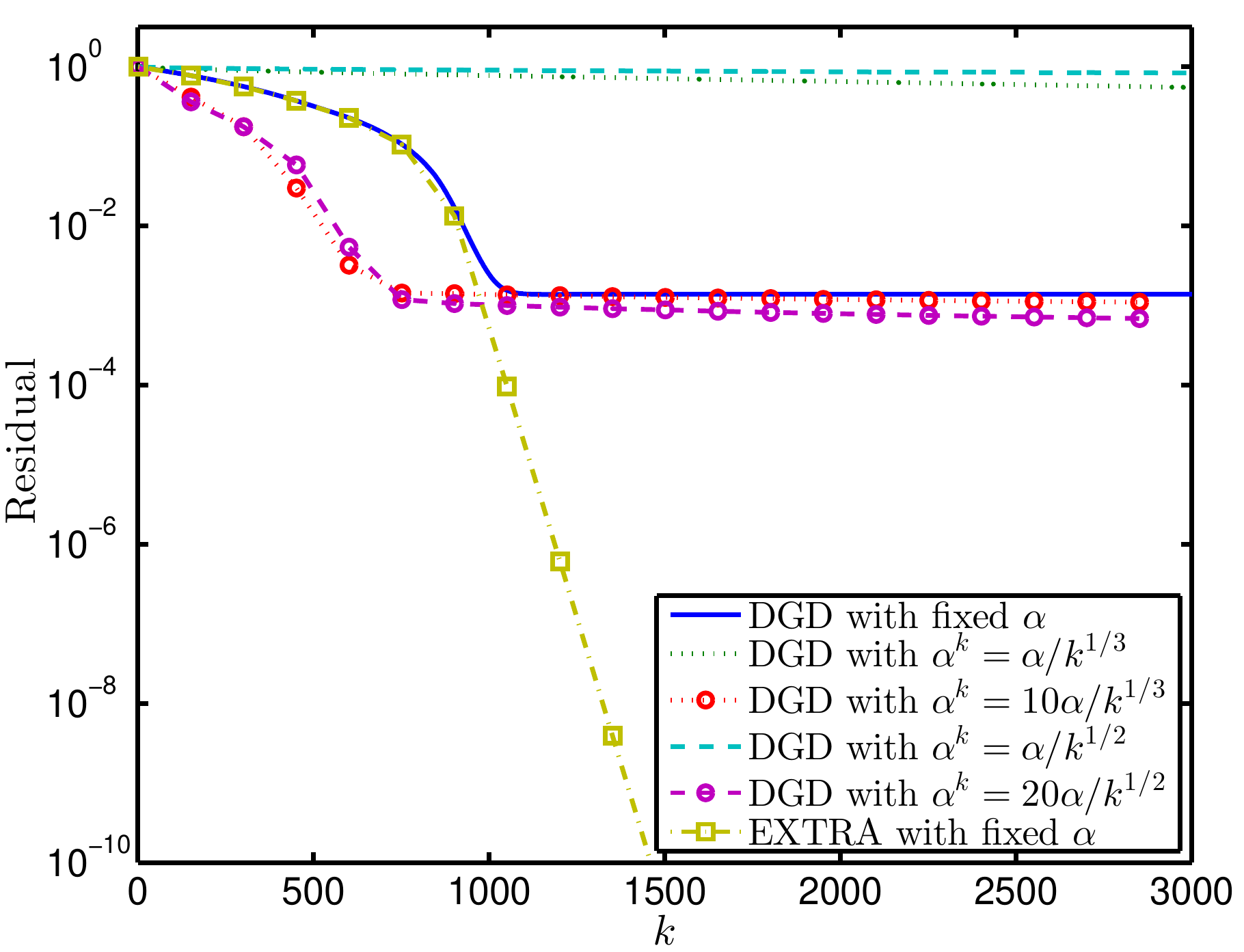}
\caption{Plot of residuals  $\frac{\|\bx^k-\bx^*\|_\Fro}{\|\bx^0-\bx^*\|_\Fro}$. Constant $\alpha=0.5276$ is the theoretical critical step size given for DGD in \cite{Kun2014}. For DGD with diminishing step sizes $O(1/k^{1/3})$ and $O(1/k^{1/2})$, we have \emph{hand-optimized} their initial step sizes as $10\alpha$ and $20\alpha$, respectively. The initial large step sizes have helped them (the red and purple curves) realize faster convergence initially.
}\label{eps:num_Huber}
\end{center}
\end{figure}

\subsection{Decentralized Logistic Regression}
Consider the decentralized logistic regression problem:
$$
\Min\limits_x \bar{f}(x)=\frac{1}{n}\sum\limits_{i=1}^n\left\{\frac{1}{m_i}\sum\limits_{j=1}^{m_i}
\ln\left(1+\exp\left(-(M_{(i)j}x)y_{(i)j}\right)\right)\right\},
$$
where every agent $i$ holds its training date $\left(M_{(i)j},y_{(i)j}\right)\in\R^p\times\{-1,+1\},\ j=1,\cdots,m_i$, including explanatory/feature variables $M_{(i)j}$ and binary output/outcome $y_{(i)j}$. To simplify the notation, we set the last entry of every $M_{(i)j}$ to $1$ thus the last entry of $x$ will yield the offset parameter of the logistic regression model.

We show a decentralized logistic regression problem solved by DGD and EXTRA over a medium-scale network. The settings
are as follows. The connected network is randomly generated with
$n=200$ agents and connectivity ratio $r=0.2$. Each agent holds $10$ samples, i.e., $m_i=10, \forall i$. The agents shall collaboratively obtain
$p=20$ coefficients via logistic regression. All the $2000$
samples are randomly generated, and the reference (ground true)
logistic classifier $x^*$ is pre-computed with a centralized
method. As it is easy to implement in practice, we use the Metropolis constant edge weight matrix $W$, which is mentioned by strategy (iii) in Section \ref{sec:matrices}, with $\epsilon=1$, and we use
$\BW=\frac{I+W}{2}$. The numerical results are illustrated in Fig.
\ref{eps:num_Logistic}. EXTRA outperforms DGD, showing linear and
exact convergence to the reference logistic classifier $x^*$.
\begin{figure}[H]
\begin{center}
\includegraphics[height=4.8cm]{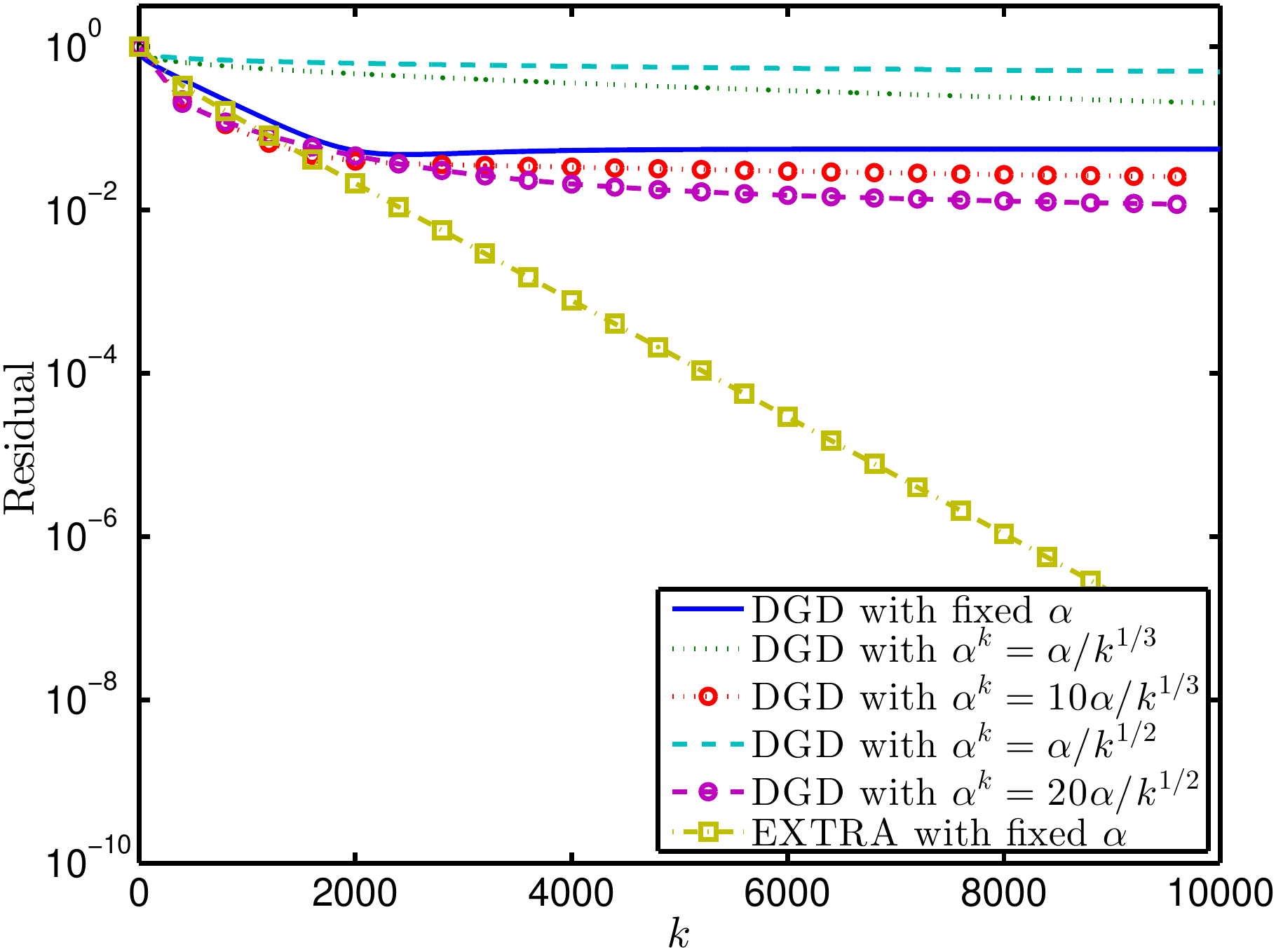}
\caption{Plot of residuals $\frac{\|\bx^k-\bx^*\|_\Fro}{\|\bx^0-\bx^*\|_\Fro}$. Constant $\alpha=0.0059$ is the theoretical critical step size given for DGD in \cite{Kun2014}. For DGD with diminishing step sizes $O(1/k^{1/3})$ and $O(1/k^{1/2})$, we have \emph{hand-optimized} their initial step sizes as $10\alpha$ and $20\alpha$, respectively.}\label{eps:num_Logistic}
\end{center}
\end{figure}

\section{Conclusion}\label{sc:concl}


As one of the fundamental method, gradient descent has been adapted to decentralized optimization, giving rise to simple and elegant iterations. In this paper, we attempted to address a dilemma or deficiency of the current decentralized gradient descent method: to obtain an accurate solution, it works slowly as it must use a small step size or iteratively diminish the step size; a large step size will lead to faster convergence  to, however, an  inaccurate solution. Our solution is an \uline{ex}act firs\uline{t}-orde\uline{r} \uline{a}lgorithm, EXTRA, which uses a fixed large step size and quickly returns an accurate solution. The claim is supported by both theoretical convergence and preliminary numerical results. On the other hand, EXTRA is far from perfect, 
and more work is needed to adapt it to the asynchronous and dynamic network settings. They are interesting open questions for future work.
\appendix
\section{Proof of Proposition \ref{prop:strongly_convex}}\label{sec:str_cvx_proof}
\begin{proof}\\

\vspace{-1.5em}

``(ii) $\Rightarrow$ (i)'': By definition of restricted strong convexity, there exists $\mug>0$ so that for any $\bx$,
\begin{equation} \label{eq:ass_equivalence_proof_1_1}
\begin{array}{rcl}
\mug \|\bx - \bx^*\|_{\Fro}^2 &\leq& \langle \dg(\bx) - \dg(\bx^*),\bx-\bx^*\rangle\\
&=& \langle \df(\bx) - \df(\bx^*),\bx-\bx^*\rangle + \frac{1}{2\alpha}\|\bx - \bx^*\|_{\BW-\W}^2.
\end{array}
\end{equation}
For any $x\in \R^p$, set $\bx = \one x^T$,  and from the above inequality, we get
\begin{equation} \label{eq:ass_equivalence_proof_1_2}
\begin{array}{rcl}
\mug \|x - x^*\|_2^2 &\leq& \frac{1}{n}\sum\limits_{i=1}^{n}\langle \nabla f_i(x) - \nabla f_i(x^*),x-x^*\rangle\\
&=&\langle \nabla \bar{f}(x) - \nabla \bar{f}(x^*), x-x^*\rangle.
\end{array}
\end{equation}
Therefore, $\bar{f}(x)$ is restricted strongly convex with a constant $\mubarf\triangleq\mug$.

``(i) $\Rightarrow$ (ii)'': For any $\bx\in\R^{n\times p}$, decompose $$\bx = \bu+\bv$$ so that every column of $\bu$ belongs to $\spa{\one}$ (i.e., $\bu$ is consensual) while that of $\bv$ belongs to $\spa{\one}^\perp$. Such an \emph{orthogonal} decomposition obviously satisfies $\|\bx\|_\Fro^2 = \|\bu\|_\Fro^2+\|\bv\|_\Fro^2$.  Since  solution $\bx^*$ is consensual and thus $\langle\bu-\bx^*,\bv \rangle = 0$, we also have $\|\bx-\bx^*\|_\Fro^2 = \|\bu-\bx^*\|_\Fro^2+\|\bv\|_\Fro^2$. In addition, being consensual, $\bu = \one u^T$ for some $u\in\R^p$. From the inequalities
\begin{eqnarray*}
\langle \df(\bu) - \df(\bx^*), \bu - \bx^* \rangle & = & n\frac{1}{n}\sum\limits_{i=1}^{n}\langle \nabla f_i(u) - \nabla f_i(x^*),u-x^*\rangle\\ & \geq & n\mubarf \|u - x^*\|^2_2=\mubarf\|\bu-\bx^*\|_\Fro^2,\\
\langle \df(\bx) - \df(\bu),\bx-\bu \rangle &\ge & 0,\\
\langle \df(\bu) - \df(\bx^*), \bx-\bu \rangle   &\ge& -\Lf\|\bu-\bx^*\|_\Fro\|\bv\|_\Fro,\\  \langle \df(\bx) - \df(\bu), \bu-\bx^* \rangle &\ge& -\Lf\|\bv\|_\Fro\|\bu-\bx^*\|_\Fro, \end{eqnarray*}
we get
\begin{equation} \label{eq:ass_equivalence_proof_2_1}
\begin{array}{rcl}
&    &\langle \df(\bx) - \df(\bx^*), \bx - \bx^* \rangle\\
&  = & \langle \df(\bu) - \df(\bx^*), \bu - \bx^* \rangle + \langle \df(\bx) - \df(\bu), \bx-\bu \rangle \\
&    &+\langle \df(\bu) - \df(\bx^*), \bx-\bu \rangle  +  \langle \df(\bx) - \df(\bu), \bu-\bx^* \rangle\\
&\geq&\mubarf\|\bu-\bx^*\|_\Fro^2 - 2\Lf\|\bu-\bx^*\|_\Fro\|\bv\|_\Fro.
\end{array}
\end{equation}
In addition, from the fact that $\bu-\bx^*\in\nul{\BW-W}$ and $\bv\in\spa{\BW-W}$, it follows that
\begin{equation} \label{eq:ass_equivalence_proof_2_2}
\begin{array}{rcl}
\frac{1}{2\alpha}\|\bx-\bx^*\|_{\BW-W}^2=\frac{1}{2\alpha}\|\bv\|_{\BW-W}^2\geq\frac{\tilde\lambda_{\min}(\BW-W)}{2\alpha}\|\bv\|_{\Fro}^2,
\end{array}
\end{equation}
where $\tilde\lambda_{\min}(\cdot)$ gives the smallest \emph{nonzero} eigenvalue of a positive semidefinite matrix.

Pick any $\gamma>0$. When $\|\bv\|_\Fro\leq\gamma\|\bu-\bx^*\|_\Fro$, it follows that
\begin{equation} \label{eq:ass_equivalence_proof_2_3}
\begin{array}{rcl}
&    &\langle \dg(\bx) - \dg(\bx^*), \bx - \bx^* \rangle\\
&  = &\langle \df(\bx) - \df(\bx^*),\bx-\bx^*\rangle + \frac{1}{2\alpha}\|\bx - \bx^*\|_{\BW-\W}^2\\
&\geq&\mubarf\|\bu-\bx^*\|_\Fro^2 - 2\Lf\|\bu-\bx^*\|_\Fro\|\bv\|_\Fro+\frac{\tilde\lambda_{\min}(\BW-W)}{2\alpha}\|\bv\|_{\Fro}^2\quad\text{(by \eqref{eq:ass_equivalence_proof_2_1} and \eqref{eq:ass_equivalence_proof_2_2})}\\
&\geq&(\mubarf-2\Lf\gamma)\|\bu-\bx^*\|_\Fro^2+\frac{\tilde\lambda_{\min}(\BW-W)}{2\alpha}\|\bv\|_{\Fro}^2\\
&\geq&\min\left\{\mubarf-2\Lf\gamma,\frac{\tilde\lambda_{\min}(\BW-W)}{2\alpha}\right\}\|\bx-\bx^*\|_\Fro^2.
\end{array}
\end{equation}
When $\|\bv\|_\Fro\geq\gamma\|\bu-\bx^*\|_\Fro$, it follows that
\begin{equation} \label{eq:ass_equivalence_proof_2_4}
\begin{array}{rcl}
&    &\langle \dg(\bx) - \dg(\bx^*), \bx - \bx^* \rangle\\
& =  &\langle \df(\bx) - \df(\bx^*), \bx - \bx^* \rangle+\frac{1}{2\alpha}\|\bx-\bx^*\|_{\BW-W}^2\\
&\geq&0+\frac{\tilde\lambda_{\min}(\BW-W)}{2\alpha}\|\bv\|_{\Fro}^2\quad\text{(applied convexity of $\f$ and \eqref{eq:ass_equivalence_proof_2_2})}\\
&\geq&\frac{\tilde\lambda_{\min}(\BW-W)}{2\alpha(1+\frac{1}{\gamma^2})}\|\bv\|_\Fro^2+\frac{\tilde\lambda_{\min}(\BW-W)}{2\alpha(1+\frac{1}{\gamma^2})}\|\bu-\bx^*\|_\Fro^2\\
&  = &\frac{\tilde\lambda_{\min}(\BW-W)}{2\alpha(1+\frac{1}{\gamma^2})}\|\bx-\bx^*\|_\Fro^2.
\end{array}
\end{equation}
Finally, in all conditions,
\begin{equation} \label{eq:ass_equivalence_proof_2_5}
\begin{array}{rcl}
&    &\langle \dg(\bx) - \dg(\bx^*), \bx - \bx^* \rangle\\ &\geq&\min\left\{\mubarf-2\Lf\gamma,\frac{\tilde\lambda_{\min}(\BW-W)}{2\alpha(1+\frac{1}{\gamma^2})}\right\}\|\bx-\bx^*\|_\Fro^2\triangleq \mug\|\bx-\bx^*\|_\Fro^2.
\end{array}
\end{equation}
By, for example, setting $\gamma=\frac{\mubarf}{4\Lf}$, we have $\mug>0$. Hence, function $\g$ is restricted strongly convex for any $\alpha>0$ as long as function $\bar{f}$ is restricted strongly convex.\hfill
\end{proof}

In the direction of  ``(ii) $\Rightarrow$ (i)'', we find $\mug<\mubarf$, unlike the more pleasant $\mubarf=\mug$ in the other direction. However, from \eqref{eq:ass_equivalence_proof_2_5}, we have
$$
\sup\limits_{\gamma,\alpha}\mug=\lim_{\gamma\rightarrow0^+}\mug\Big|_{\alpha=\frac{\tilde\lambda_{\min}(\BW-W)}{2(1+\frac{1}{\gamma^2})(\mubarf-2\Lf\gamma)}}=\mubarf,
$$
which means that $\mug$ can be arbitrarily close to $\mubarf$ as $\alpha$ goes to zero. On the other hand, just to have $O(\mug)=O(\mubarf)$, we can set $\gamma=O\left(\frac{\mubarf}{\Lf}\right)$ and $\alpha=\frac{\tilde\lambda_{\min}(\BW-W)}{2(1+\frac{1}{\gamma^2})(\mubarf-2\Lf\gamma)}=O\left(\frac{\mubarf}{\Lf^2}\right)=O\left(\frac{\mug}{\Lf^2}\right)$. This order of $\alpha$ coincides, in terms of order of magnitude,  with the critical step size for ensuring the linear convergence.


\bibliographystyle{siam}
\bibliography{document}
\end{document}